\documentclass[12pt,reqno]{amsart}
\usepackage{fullpage}
\usepackage{times}
\usepackage{amsmath,amssymb,amsthm,url}
\usepackage[utf8]{inputenc}
\usepackage[english]{babel}
\usepackage{comment}
\usepackage{bbm}
\usepackage{enumerate}

\usepackage{bm}
\usepackage{graphicx}
\usepackage{mathrsfs}
\usepackage[colorlinks=true, pdfstartview=FitH, linkcolor=blue, citecolor=blue, urlcolor=blue]{hyperref}
\usepackage[vlined,ruled]{algorithm2e}
\usepackage{mathtools}

\newtheorem{thm}{Theorem}[section]
\newtheorem{lem}[thm]{Lemma}
\newtheorem{prop}[thm]{Proposition}
\newtheorem{cor}[thm]{Corollary}
\newtheorem{question}[thm]{Question}
\newtheorem{conj}[thm]{Conjecture}
\usepackage[vlined,ruled]{algorithm2e}

\theoremstyle{definition}

\newtheorem{defn}[thm]{Definition}

\newtheorem{rem}[thm]{Remark}

\newcommand{\N}{\mathbb{N}}
\newcommand{\Z}{\mathbb{Z}}
\newcommand{\rad}{\operatorname{rad}}
\newcommand{\lcm}{\operatorname{lcm}}
\newcommand{\eps}{\varepsilon}
\newcommand{\PP}{\operatorname{PP}}

\title{Hilbert cubes in sets with arithmetic properties}
\author{Ernie Croot}
\address{School of Mathematics\\ Georgia Institute of Technology\\ Atlanta, GA 30332\\ United States}
\email{ernest.croot@math.gatech.edu}
\author{Junzhe Mao}
\address{School of Mathematics\\ Georgia Institute of Technology\\ Atlanta, GA 30332\\ United States}
\email{jmao87@gatech.edu}
\author{Chi Hoi Yip}
\address{School of Mathematics\\ Georgia Institute of Technology\\ Atlanta, GA 30332\\ United States}
\email{cyip30@gatech.edu}
\subjclass[2020]{Primary 11B30, 11P70; Secondary 11B25, 11N36}
\keywords{Hilbert cube, subset sum, arithmetic progression, sumset, powerful number, prime, perfect power, smooth number}

\begin{document}

\begin{abstract}
In this paper, we introduce new general frameworks for estimating the maximal dimension of Hilbert cubes contained in finite truncations of arbitrary sets. As applications, we investigate Hilbert cubes in a range of arithmetic sets, including perfect powers, powerful numbers, primes, smooth numbers, and squarefree numbers. Along the way, we substantially sharpen several earlier results of Dietmann--Elshotlz, Erd\H os--S\'ark\"ozy--Stewart, Hajdu, and S\'ark\"ozy, and we obtain bounds that are sharp up to the implied constant in several cases. Additionally, we prove conditional results of independent interest, including an almost sharp uniform upper bound on the number of $k$-th powers in an arithmetic progression for each $k\geq 4$, assuming the ABC conjecture.
\end{abstract} 

\maketitle

\section{Introduction}
Let $G$ be a finite additive group. Let $a_1, a_2, \ldots, a_d \in G \setminus \{0\}$, where $a_1,a_2,\ldots, a_d$ are not necessarily distinct. Let $A=\{a_1,a_2,\ldots, a_d\}$ be a multiset, the \emph{set of sub(multi)set sums} of $A$ is defined to be
$$
\Sigma(A)=\bigg\{\sum_{i=1}^d \eps_i a_i: \eps_i \in \{0,1\}\bigg\},
$$
sometimes we exclude the empty sum (which is defined to be $0$) and instead consider
$$
\Sigma^*(A)=\bigg\{\sum_{i=1}^d \eps_i a_i: \eps_i \in \{0,1\}, \sum_{i=1}^d \eps_i>0\bigg\}.
$$
If $a_0 \in G$ and $a_1, a_2, \ldots, a_d \in G \setminus \{0\}$, we define the \emph{Hilbert cube}
$$
H(a_0;a_1,a_2,\ldots, a_d)=a_0+\{0, a_1\}+\{0,a_2\}+\cdots+\{0,a_d\}=a_0+\Sigma(A);
$$
when $a_0=0$, we instead define $H(0;a_1,a_2,\ldots, a_d)=\Sigma^*(A)$
to exclude the empty sum.

Exploring the arithmetic structure of subset sums and Hilbert cubes plays a fundamental role in arithmetic combinatorics and number theory. For example, the connection between Hilbert cubes and arithmetic progressions has been famously studied by Szemer{\'e}di~\cite{S75}, Gowers \cite{Gowers01}, Szemer{\'e}di--Vu~\cite{SV06, SV06a}, and Green--Tao \cite{GT08}. 

One important question in the study of Hilbert cubes is the following.

\begin{question}\label{question:main}
Let $R$ be a finite subset of positive integers (often with some arithmetic structure). What is the maximum possible dimension of a Hilbert cube in $R$? More precisely, what is the largest integer $d$, such that there exist a nonnegative integer $a_0$, and positive integers $a_1,a_2,\ldots, a_d$, such that $H(a_0;a_1,a_2,\ldots, a_d)\subseteq R$? What if we additionally assume that $a_1,a_2, \ldots, a_d$ are distinct?
\end{question}

Question~\ref{question:main} has been studied extensively. For example, we refer to the following (listed in chronological order):  Pomerance--S\'ark\"ozy--Stewart \cite{PSS88}, Erd\H os--S\'ark\"ozy--Stewart \cite{ESS94}, Hegyv\'ari--S\'ark\"ozy~\cite{HS99}, Hegyv\'ari~\cite{H99}, Gyarmati--S\'ark\"ozy--Stewart \cite{GSS03}, and Dietmann--Elshotlz \cite{DE12, DE15}.

In this paper, we develop new frameworks to study Question~\ref{question:main}. As applications, we focus on the case that $R=S\cap [N]$, where $S$ is a given infinite arithmetic set of special interest, and we would like to predict how $d$ grows asymptotically with $N$ \footnote{Alternatively, we may assume $R$ is an infinite subset of positive integer and impose the assumptions that $a_1,\ldots, a_d\in [N]$ (sometimes we also assume that $a_0\in [N]$).}. Moreover, for some $S$, the upper bound on the dimension produced from our new frameworks matches (in magnitude) with the lower bound from the \emph{arithmetic progression construction}, and is thus sharp (up to the implied constant). Arithmetic progressions are special Hilbert cubes and thus studying the longest possible length of an arithmetic progression contained in $R$ yields lower bounds for Question~\ref{question:main}. More precisely, if $q,\ell$ are positive integers such that $\{a_0+q, a_0+2q, \ldots, a_0+\ell q\}$ is an arithmetic progression contained in $R$, then the multiset $A$ consisting of the element $q$ with multiplicity $\ell$ satisfies that $a_0+\Sigma^*(A)=\{a_0+q, a_0+2q, \ldots, a_0+\ell q\} \subseteq R$, and the set $A'=\{q, 2q, \ldots, \lfloor \sqrt{2\ell} \rfloor q\}$ satisfies that $a_0+\Sigma^*(A')=\{a_0+q, a_0+2q, \ldots, a_0+\ell q\} \subseteq R$. This justifies the term \emph{arithmetic progression construction}.

In the next two subsections, we motivate and describe our new frameworks. We then discuss applications of these frameworks to Hilbert cubes contained in specific arithmetic sets of interest in Section~\ref{sec:app}. In particular, improving substantially on several previous results, we study the cases where $S$ is given by each of the following sets: perfect powers, powerful numbers, primes, smooth numbers, and squarefree numbers. Our proofs combine various ideas from arithmetic combinatorics, Diophantine equations, and sieve methods. 

\subsection{The first framework}
Let $S$ be a fixed infinite set. Our goal is to bound the dimension of Hilbert cubes contained in $S\cap [N]$. Our first framework utilizes the information of arithmetic progressions (measured by $G(N)$, the longest length of an arithmetic progression in $S\cap [N]$) and sumsets contained in a set $S$. This can be measured by the function $F_2(N)$: consider all the sumsets $A+B\subseteq S$ with $A,B\subseteq [N]$ and $|A|\leq |B|$, we define $F_2(N)$ to be the maximum size of such $A$. By writing a Hilbert cube contained in $S\cap [N]$ as a sumset, it follows that its dimension is at most $ 2F_2(N)+1$. More generally, for $k$-fold sumsets, we define a similar function $F_k(N)$ (see the precise definition in the theorem below) and it is easy to see that the dimension of a Hilbert cube contained in $S\cap [N]$ is at most $kF_k(N)+k-1$.

These two functions $G(N)$ and $F_2(N)$ each apply to two opposite extremes for Hilbert cubes contained in $S\cap [N]$, where on one extreme such a Hilbert cube can be a long arithmetic progression (making $G(N)$ large), and on the other extreme it could be a set of distinct subset sums $\Sigma^*(A)$ (making $F_2(N)$ large). As discussed above, the maximal dimension of a Hilbert cube in $S\cap [N]$ lies in the interval $[G(N), 2F_2(N)+1]$. While it is not obvious how $G(N)$ can be leveraged to give an upper bound on the maximal dimension, our theorem demonstrates that the dimension can be ``essentially bounded" by $G(N)$.

\begin{thm}\label{main}
Let $S\subset \N$ and $k\geq 2$ be a positive integer. For each $N\in \N$, define the following:
\begin{enumerate}[(1)]
    \item $F_k(N)$ is the smallest integer at least $2$ such that whenever $B_1,B_2,\ldots,B_k$ are subsets of $[N]$ such that $\sum_{i\in I}B_i\subseteq S$ for all nonempty subsets $I$ of $[k]$, we have $$\min_{1\leq i \leq k}|B_i|\leq F_k(N).$$
    \item $G(N)$ is the longest possible length of an arithmetic progression contained in $S \cap [N]$.
\end{enumerate}
If $a_0$ is a nonnegative integer, and $A \subseteq \N$ is a multiset such that  $a_0+\Sigma^*(A)\subseteq S\cap [N]$, then 
\begin{enumerate}
    \item $|A|\ll k G(N)\log F_k(N)$, where the implied constant is absolute. 
    \item Assume that $F_k(N) \to \infty$ as $N \to \infty$. For any $\eps>0$, we have $$|A|\ll_{k,\eps}\max\{F_k(N)^\eps,G(N)\}.$$
    \item Assume that $F_k(N) \to \infty$ as $N \to \infty$, and additionally that $A$ is a set. Then for any $\eps>0,$ there exists some $\delta=\delta(k,\eps)\in(0,1)$ such that $$ 
    |A|\ll_{k,\eps}\max\{F_k(N)^\eps,G(N)^{1-\delta}\}.$$
\end{enumerate}
\end{thm}

We refer to Theorem~\ref{main_strong} for a more general version of Theorem~\ref{main}.
This approach is especially powerful when applied to sets $S$ that exhibit no rich additive structure. Indeed, for most of our applications in this paper, $S$ is a multiplicatively defined set; therefore, one generally expects $S$ to lack a rich additive structure due to the sum-product phenomenon. For these sets $S$, we expect that additive patterns in $S$ are essentially ``controlled'' by arithmetic progressions in $S$, and the maximal dimension of a Hilbert cube contained in $S\cap [N]$ should have the same magnitude as $G(N)$. 

\begin{rem}\label{rem:framework1}
We compare the above bounds and discuss the sharpness of part (b).
\begin{enumerate}[(1)]
    \item Some special cases of Theorem~\ref{main}(a) have appeared implicitly in Dietmann and Elsholtz~\cite{DE15}. In particular,  Theorem~\ref{main}(a)    
    strengthens their result~\cite[Theorem 1.3]{DE15}, where they showed that $|A|\ll G(N)\log N$ under the same condition. 
 
\item If there is a fixed $\eps>0$ such that $F_k(N)^{\eps}\ll G(N)$, equivalently, $\log F_k(N)/ \log G(N)$ is bounded above, then Theorem~\ref{main}(b) implies that $|A|\ll G(N)$. This bound is sharp up to the implied constant, as demonstrated by taking $S$ to be an infinite arithmetic progression. Beyond this trivial case, we will see in Theorem~\ref{thm:smooth1} that Theorem~\ref{main}(b) is nontrivially sharp for sets of certain smooth numbers. We also show in Remark~\ref{rem:random} that it is sharp for a random set almost surely.

%\item %We expect that for most sets $S$, the arithmetic progression construction is essentially the only way to produce the maximum Hilbert cubes in $S\cap[N]$; more precisely, all Hilbert cubes in $S\cap[N]$ have dimension bounded above by $CG(N)$ for some constant $C$ depending on $S$. Next, we briefly discuss a random set model to illustrate this.

\item It would be desirable to improve the bound $G(N)^{1-\delta}$ in Theorem~\ref{main}(c) to $\sqrt{G(N)}$ in view of the lower bound from the arithmetic progression construction. Nevertheless, Theorem~\ref{main}(c) provides a nontrivial power saving on the bound $G(N)$ compared to the multiset setting. 

\item For most sets $S$ (including perfect powers, powerful numbers, primes, smooth numbers, and squarefree numbers) we consider in this paper, since they possess a strong multiplicative property but no obvious additive structure, we expect that $F_k$ and $G$ in Theorem~\ref{main} satisfy that $F_k(N)=G(N)^{O(1)}$ and $G(N)=(\log N)^{O(1)}$, so that Theorem~\ref{main}(b) would be stronger compared to Theorem~\ref{main}(a) and might lead to a sharp bound.  
\end{enumerate}

\end{rem}

\subsection{The second framework}\label{subsec:second}
Our first framework is powerful, however, for many arithmetic sets $S$, the current known techniques are not sufficient to provide a reasonably good bound on $F_k(N)$ (in fact, typically the best known bound on $F_k(N)$ is of the form $N^{O(1)}$, for a fixed $k$). This barrier motivates us to develop our second framework, which employs an approach that avoids $F_k(N)$. 

Many arithmetic sets $S$ are characterized by their $q$-divisibility for various primes or prime powers $q.$ This naturally motivates the use of sieve methods to study Hilbert cubes in $S\cap[N].$ Our second framework is devoted to combining the larger sieve with tools from the theory of zero-sum-free and incomplete sequences in finite abelian groups. Although we formulate the core results in terms of subset sums, we will generalize their application to the setting of Hilbert cubes. We also note that, as shown in subsequent sections, several results derived from this framework are optimal. 

Below, we present several theorems that play the role of Theorem~\ref{main} in this new framework, providing upper bounds on $|A|$ using an auxiliary set $B$ associated to $\Sigma^*(A)$. This idea was inspired by Erd\H os, S\'ark\"ozy, and Stewart \cite{ESS94}. Our following theorem has the same setting as their paper.

\begin{thm}\label{thm:primefactor}
There exists an absolute constant $c>0$, such that if $A \subseteq[N]$ is a set and $B$ is a subset of $[1,c|A|^2]$ consisting of pairwise coprime integers none of which divides a member of $\Sigma^*(A)$, then
$$
|A|\sum_{b\in B} \frac{\log b}{b} \ll \log N.
$$
\end{thm}

This provides a uniform treatment of \cite[Theorem 1 and 2]{ESS94} in their paper. In particular, it implies \cite[Theorem 1]{ESS94} as an immediate corollary: Under the same assumption as in Theorem~\ref{thm:primefactor}, 
\begin{equation}\label{eq:ESSthm1}
\sum_{b\in B} \frac{\log b}{\sqrt{b}}\ll \log N.
\end{equation} 
Indeed, inequality~\eqref{eq:ESSthm1} follows from Theorem~\ref{thm:primefactor} since $b^{1/2}\ll |A|$ for all $b\in B$. It also refines \cite[Theorem 2]{ESS94}, where they showed a weaker bound $|A| \sum_{b\in B} 1/b\ll \log N$ under the stronger assumption that $B \subseteq [1, C|A|^2/\log |A|]$ (where $C$ is an absolute constant). 

We remark that Theorem~\ref{thm:primefactor} is optimal up to the constant $c$ by considering the following example. If $C$ is a sufficiently large constant, $A=[\lfloor\log N\rfloor]$, and $B$ is the set of primes in $[2(\log N)^2, C(\log N)^2]\subseteq [1, C|A|^2]$, then none of the elements in $B$ divides a member of $\Sigma^*(A)$, and we have
$$
|A|\sum_{b\in B} \frac{\log b}{b}\geq \frac{\log C \cdot\log N}{2}
$$
for all sufficiently large $N$. 

We also provide the following strengthening of Theorem~\ref{thm:primefactor} when $B$ consists of prime powers that are pairwise coprime. 

\begin{thm}\label{thm:main4}
For any $\eps>0,$ there exists a constant $c=c(\eps)>0$, such that if $A \subseteq[N]$ is a set and $B$ is a subset of $[1,c|A|^2]$ consisting of pairwise coprime prime powers, such that for each $b\in B$, the members of $\Sigma^*(A)$ do not run over all residue classes modulo $b$, then 
$$
|A|\sum_{b\in B'} \frac{\log b}{b} \ll_\eps \log N
$$
provided that $B'=B \setminus [1,X]$, where $X$ is the smallest integer such that $\sum_{b\in B, b\leq X} \Lambda(b)>(1+\eps)\log N$.
\end{thm}

For our applications to Hilbert cubes, Theorem~\ref{thm:main4} is more flexible compared to Theorem~\ref{thm:primefactor}. One key ingredient of our proofs are new results about zero-sum-free sequences (Theorem~\ref{thm:Aq}) and incomplete sequences (Theorem~\ref{thm:incom}) in finite abelian groups, which are of independent interest. We also prove a multiset version of Theorem~\ref{thm:main4}.

\begin{thm}\label{thm:multi_incom}
Let $\eps>0$. Let $A \subseteq[N]$ be a multiset and $B$ be a set consisting of pairwise coprime prime powers, such that for each $b\in B$, the members of $\Sigma^*(A)$ do not run over all residue classes modulo $b$. 
If $\sum_{b\in B} \Lambda(b)>(1+\eps)\log N$, then $|A|\leq (1+\eps^{-1})\max_{b\in B} b$.

\end{thm}

\medskip

\subsection{Applications of new frameworks} In the next section, we discuss some applications of new frameworks to Hilbert cubes in arithmetic sets in detail. We improve several results in the literature, and some of our new results are almost sharp. In particular, we employ our first framework to study Hilbert cubes in perfect powers, 
our second framework to Hilbert cubes in powerful numbers, and a mixture of the two frameworks to Hilbert cubes in smooth numbers. We also discuss some quick applications of our second framework to Hilbert cubes in primes and squarefree numbers.

\medskip

\textbf{Notation.} We follow standard notation in arithmetic combinatorics and analytic number theory. In this paper,~$p$ always denotes a prime, and $\sum_p$ and $\prod_p$ represent sums and products over all primes; 
$N$ always denotes a positive integer and $[N]$ denotes the set $\{1,2,\ldots, N\}$. 
Given two sets $A,B \subseteq \Z$, we write $A+B=\{a+b: a \in A, b \in B\}$ and $AB=\{ab: a \in A, b \in B\}$. We use the Vinogradov notation $\ll$; we write $X \ll Y$ or $Y\gg X$ if there is an absolute constant $C>0$ so that $|X| \leq CY$. For a positive integer $q$, we use $\Z_q$ to denote the cyclic group $\Z/q\Z$. We use $\Lambda$ to denote the von Mangoldt function.

\textbf{Organization of the paper.} 
In Section~\ref{sec:app}, we discuss applications of our new frameworks to Hilbert cubes in various arithmetic sets of interest. In Section~\ref{sec:1stframework}, we prove Theorem~\ref{main}, our first framework. In Section~\ref{sec:powerful}, we discuss Hilbert cubes in powerful numbers. In Section~\ref{sec:second}, we prove the results within our second framework. In Section~\ref{sec:perfectpower}, we study Hilbert cubes in perfect powers. Finally, in Section~\ref{sec:smooth}, we consider Hilbert cubes in smooth numbers.

\section{Applications of new frameworks to Hilbert cubes in arithmetic sets}\label{sec:app}

%In this section, we discuss some applications of new frameworks to Hilbert cubes in arithmetic sets. In particular, we employ our first framework to study Hilbert cubes in perfect powers, 
%our second framework to Hilbert cubes in powerful numbers, and a mixture of the two frameworks to Hilbert cubes in smooth numbers. We also discuss some quick applications of our second framework to Hilbert cubes in primes and squarefree numbers.

\subsection{Hilbert cubes in powerful numbers}\label{sec:intro_powerful}

A positive integer $n$ is \emph{powerful} if $p^2 \mid n$ for each prime $p$ dividing $n$. Gyarmati, S\'ark\"ozy, and Stewart \cite{GSS03} initiated the study of a set $A \subseteq [N]$ such that $\Sigma^*(A)$ is contained in the set of powerful numbers. They showed that for maximal such $A$,
\begin{equation}\label{eq:GSS}
(\log N)^{1/2} \ll |A|\ll \frac{(\log N)^3}{(\log \log N)^{1/2}}.
\end{equation}
Dietmann and Elshotlz \cite{DE15} studied the same question for Hilbert cubes. They showed that if $a_0$ is a non-negative integer and $a_1, a_2, \ldots, a_d$ are positive integers such that $H(a_0; a_1, a_2, \ldots, a_d)$ is contained in the intersection of the set of powerful numbers and $[N]$, then $d \ll (\log N)^2$, improving the upper bound in inequality~\eqref{eq:GSS}. We further improve their results in the following more general setting, where we enlarge the set of powerful numbers to $W_{\mathcal{P}}$ defined below.

\begin{thm}\label{thm:powerful}
    Suppose $\mathcal{P}$ is a subset of primes, $a_0$ is a non-negative integer and $a_1, a_2, \ldots, a_d \in [N]$ are positive integers such that $H(a_0; a_1, a_2, \ldots, a_d)$ is contained in the set $W_{\mathcal{P}},$ where 
    \[W_{\mathcal{P}} = \{n\in\N : p \mid n \Rightarrow p^2 \mid n\ \text{for any }p \in\mathcal{P}\}.\]
    \begin{enumerate}
        \item If there exist some constants $C,C'>1$ such that 
        \begin{equation}\label{eq:powerful_cond1}
            \sum_{\substack{p\leq C\log N\\p\in \mathcal{P}}}\log p\geq C'\log N,
        \end{equation}
        then $d\ll \frac{C(C'+1)}{C'-1}\log N$, where the implied constant is absolute.
        \item If the $a_i$'s are distinct, and there exist some constants $C,C'>0$ such that 
        \begin{equation}\label{eq:powerful_cond2}
        \sum_{\substack{p\leq C\log N\\p\in \mathcal{P}}}\log p>2\log N\quad \text{and}\quad \sum_{\substack{C\log N< p\leq (\log N)^{1.5}\\p\in \mathcal{P}}}\frac{\log p}{p}\geq C'\log\log N,
        \end{equation}
        then $d\ll\frac{\log N}{C'\log\log N}$, where the implied constant is absolute.
    \end{enumerate}
\end{thm}

By the prime number theorem, taking $\mathcal{P}$ to be the set of \emph{all} primes immediately yields the following corollary for powerful numbers.
\begin{cor}\label{cor:powerful}
    Suppose $a_0$ is a non-negative integer and $a_1, a_2, \ldots, a_d \in [N]$ are positive integers such that $H(a_0; a_1, a_2, \ldots, a_d)$ is contained in the set of powerful numbers, then $d\ll \log N$. If $a_1,\ldots,a_d$ are distinct, then $d\ll \frac{\log N}{\log \log N}$.
\end{cor}

%\begin{thm}\label{thm:powerful}
%    If $a_0$ is a non-negative integer and $a_1, a_2, \ldots, a_d \in [N]$ are positive integers such that $H(a_0; a_1, a_2, \ldots, a_d)$ is contained in the set of powerful numbers, then $d\ll \log N.$ 
%\end{thm}
%We remark that for different choices of $\mathcal{P}$, the density of $W_{\mathcal{P}}$ can vary substantially. For example, when $\mathcal{P}$ is the set of all primes, $W_\mathcal{P}$ is the set of powerful numbers and has density $O(N^{-1/2})$; however, when $\mathcal{P}$ is the set of primes congruent to $3\pmod 4$, all the primes congruent to $1\pmod 4$ are contained in $W_\mathcal{P}$ and hence it has density $\Omega((\log N)^{-1})$.

Theorem~\ref{thm:powerful}(a) is sharp (up to the implied constant) because we have the following simple matching lower bound construction. Let $d$ be a positive integer and $a_0=a_1=a_2=\ldots=a_d=\prod_{p\leq d+1} p^2$. Then the Hilbert cube $$H(a_0; a_1,a_2,\ldots, a_d)=\left\{j \prod_{p\leq d+1} p^2: 1\leq j \leq d+1\right\}$$ is contained in the set of powerful numbers. By the prime number theorem, we can choose $d\gg \log N$ so that $\prod_{p\leq d+1} p^2\leq N$. 

We also note that the condition $C,C'>1$ in Theorem~\ref{thm:powerful}(a) is necessary. Indeed, if $\eps>0$ and $\mathcal{P} = \{p:p\leq (1-\eps)\log N\}$ with $N$ sufficiently large, then Hilbert cubes contained in $W_{\mathcal{P}}$ do not have bounded dimension. For example, for each $m\in \N$, $W_{\mathcal{P}}$
contains the Hilbert cube $1+\Sigma^*(A)$, where $A$ is the multiset consisting of the element $a$ with multiplicity $m$, where
\[a = \prod_{p\leq (1-\eps)\log N} p.\]

%We also improve Theorem~\ref{thm:powerful} slightly under the extra assumption that $a_1,a_2,\ldots, a_d$ are distinct.

%\begin{thm}\label{thm:powerful2}
%    If $a_0$ is a non-negative integer and $a_1, a_2, \ldots, a_d \in [N]$ are distinct positive integers such that $H(a_0; a_1, a_2, \ldots, a_d)$ is contained in the set of powerful numbers, then $d\ll \frac{\log N}{\log \log N}.$ 
%\end{thm}

The proof of Theorem~\ref{thm:powerful} is based on similar considerations of local versions (that is, modulo $p^2$) from our second framework, but our second framework does not apply directly. In fact, our second framework is inspired by the proof of Theorem~\ref{thm:powerful}. 

\subsection{Hilbert cubes in perfect powers}\label{sec:intro_perfectpower}
In this section, we use our first framework to study Hilbert cubes in perfect powers. For convenience, we use $\operatorname{PP}$ to denote the set of perfect powers.

To apply the first framework, a key step is to bound the length of an arithmetic progression consisting of perfect powers. Euler \cite[p.21]{M69} showed that there is no 4-term arithmetic progression consisting of squares, and 
Darmon and Merel \cite{DM97} showed that there is no 3-term arithmetic progression consisting of $k$-th powers if $k\geq 3$. Arithmetic progressions consisting of perfect powers have also been studied \cite{DE15, GSS03, H04}. The best-known result, proved by Dietmann and Elshotlz \cite{DE15}, is summarized in the following lemma.

\begin{lem}\label{lem:APperfectpower}
\
\begin{enumerate}
    \item (\cite[Lemma 4.8]{DE15}) The maximum length of a homogeneous arithmetic progression contained in $\PP \cap [N]$ is $\asymp \frac{\log \log N}{\log \log \log N}$.
    \item (\cite[Section 3]{DE15}) If $a,q$ are positive integers such that the arithmetic progression $\{a+qj: 0\leq j \leq \ell-1\}$ is contained in $\PP \cap [N]$, then $\ell \ll \log N$.  
\end{enumerate}
\end{lem}

Inhomogeneous arithmetic progressions consisting of perfect powers are much harder to study. In particular, Lemma~\ref{lem:APperfectpower}(b) is believed to be far from sharp. Hajdu \cite{H04} showed that, under the ABC conjecture, the bound in (b) of the above lemma can be bounded in terms of $\gcd(a,q)$; in fact, his proof led to a bound of tower type in $\gcd(a,q)$. The following theorem significantly improves his result; moreover, it is sharp by Lemma~\ref{lem:APperfectpower}(a). 

\begin{thm}\label{thm:APperfectpower}
Assume the ABC conjecture. If $a,q,\ell$ are positive integers such that the arithmetic progression $\{a+jq: 0\le j \leq \ell-1\}$ is contained in $\PP$, then 
$$
\ell \ll \frac{\log \log a}{\log \log \log a},
$$
where the implied constant is absolute. 
\end{thm}

A key ingredient of our proof of Theorem~\ref{thm:APperfectpower} is to study the number of $k$-th powers (for a fixed $k$) contained in an arithmetic progression, which is of independent interest. For each $k\geq 2$, $N,q\geq 1$ and $a\in \Z$, let $Q_k(N;q,a)$ denote the number of  $k$-th powers in the arithmetic progression $a+q,a+2q, \ldots, a+N q$. Let $Q_k(N)$ be the maximum number of $k$-th powers which can appear in an arithmetic progression of length $N$, that is, $Q_k(N)=\max_{q\geq 1, a \geq 0} Q_k(N;q,a)$. By taking $a=0$ and $q=1$, it follows that $Q_k(N)\gg N^{1/k}$. It is widely believed that $Q_k(N)\ll N^{1/k+o(1)}$ for each fixed $k\geq 2$, however this conjecture is widely open. Bombieri, Granville, and Pintz \cite{BGP92} showed that $Q_2(N) \ll N^{2/3} (\log N)^{O(1)}$, and Bombieri and Zannier \cite{BZ02} further improved their result to $Q_2(N)\ll N^{3/5} (\log N)^{O(1)}$. In \cite{BGP92}, the authors remarked that they expected that their proof techniques could be adapted to show that $Q_3(N) \ll N^{3/5+o(1)}$ and $Q_k(N)\ll N^{1/2+o(1)}$ for each fixed $k\geq 4$; however, they also mentioned that there might be ``further complications" in adapting their approach to $Q_k(N)$ with $k\geq 3$. Recently, Shkredov and Solymosi~\cite{SS21} showed that under the uniformity conjecture,  $Q_3(N)\ll N^{2/3} \exp(-O((\log N)^{\alpha}))$ for each fixed $\alpha<\frac{1}{11}$, and $Q_k(N)\ll N^{1/2} \exp(-O((\log N)^{1/3}))$ for each fixed $k\geq 4$.

While we are not able to improve the above results from the literature, our next theorem gives an almost sharp upper bound on $Q_k(N)$ for $k\geq 4$, conditional on the ABC conjecture. 

\begin{thm}\label{thm:ABCQ_k}
Assume the ABC conjecture. Then there is an absolute constant $C$, such that 
$$Q_k(N)\leq N^{1/k} \exp \bigg(Ck \frac{\log N}{\log \log N}\bigg)$$
holds for all $k\geq 4$ and $N\geq 1$.
\end{thm}

\begin{rem}
\cite[Section 9] {CG07} contains some discussions related to the connection of the ABC conjecture and bounds on $Q_2(N)$. It would be interesting to see if Theorem~\ref{thm:ABCQ_k} holds for $k\in \{2,3\}$.  
\end{rem}

For Hilbert cubes contained in perfect powers, the best known results are due to Dietmann and Elsholtz. They showed in \cite[Corollary 1.6]{DE15} that if $H(a_0;a_1,a_2,\ldots, a_d)\subset [N]\cap \PP$, then $d\ll (\log N)^2$. Since perfect powers are powerful, Corollary~\ref{cor:powerful} already implies a stronger bound $d\ll \log N$. They also showed a much stronger bound for subset sums in perfect powers in \cite[Theorem 1.7]{DE15}: if $H(0;a_1,a_2,\ldots, a_d)\subset [N]\cap \PP$, then $d\ll \frac{(\log \log N)^4}{(\log \log \log N)^2}$. 

Our next result can be viewed as an extension of \cite[Theorem 1.7]{DE15}. We show that if $a_0$ has a "small" prime factor and $a_1,a_2,\ldots, a_d$ are distinct, then we can get a power saving from the $(\log N)^{1-o(1)}$ bound in Corollary~\ref{cor:powerful}.

\begin{thm}\label{thm:localperfectpower}
Let $c>0$. Let $a_0$ be a positive integer that has a prime factor $\leq (\log N)^{1-c}$, and $a_1,a_2,\ldots, a_d \in [N]$ be distinct, such that $H(a_0;a_1,a_2,\ldots, a_d)\subset \PP$. Then $d\ll_{c} (\log N)^{1-c'}$, where $c'>0$ only depends on $c$. 
\end{thm}

We also provide several conditional improvements on \cite[Theorem 1.7]{DE15} and Theorem~\ref{thm:localperfectpower}. 

\begin{thm}\label{thm:ABCperfectpowerintro}
Assume the ABC conjecture. 
\begin{enumerate}
    \item If $a_0$ is a nonnegative integer with a prime factor $p$, and 
    $a_1,a_2,\ldots, a_d \in [N]$ such that $H(a_0;a_1,a_2,\ldots, a_d)\subset [N] \cap \PP$, then 
$$
d\ll p+\frac{(\log \log N)^2}{\log \log \log N}.
$$  
In particular, if $H(0;a_1,a_2,\ldots, a_d)\subset [N]\cap \PP$, then 
$$
d\ll \frac{(\log \log N)^2}{\log \log \log N}.
$$  
\item If $H(a_0;a_1,a_2,\ldots, a_d)\subset [N]\cap \PP$, then 
$$d \ll \frac{(\log \log N)^3}{\log \log \log N}.$$
Furthermore, assuming  Conjecture~\ref{conj:LPS2}, we have the stronger bound
$$
d\ll \frac{(\log \log N)^2}{\log \log \log N}.
$$
\end{enumerate}
\end{thm}

Conjecture~\ref{conj:LPS2}, mentioned in the above theorem, states that if $k$ is sufficiently large, then the set of $k$-th powers forms a Sidon set. It is a special case of the Lander--Parkin--Selfridge conjecture~\cite{LPS67} on sums of perfect powers (see Conjecture~\ref{conj:LPS}).

\subsection{Hilbert cubes in smooth numbers}\label{sec:intro_smooth}
To improve the bound in Corollary~\ref{cor:powerful} in the setting of sets, we examine our lower bound construction more closely. For any $A_0\subseteq [\sqrt{N}]$ with $\Sigma^*(A_0)$ being $0.1\log N$-smooth, if we define
\[A = \bigg(\prod_{p\leq0.1\log N}p^2\bigg)\cdot A_0,\]
then $A\subseteq[N]$ and $\Sigma(A)$ is contained in the set of powerful numbers. The lower bound in (\ref{eq:GSS}) is obtained by taking $A_0=[0.1\sqrt{\log N}]$. A larger $A_0$ would immediately lead to an improvement on the lower bound.
This raises the natural question: What is the maximum size of $A\subseteq[\sqrt{N}]$ such that $\Sigma^*(A)$ is contained in $0.1\log N$-smooth numbers? More generally, given some function $y:\N\rightarrow\N$, what is the maximum size of $A\subseteq[N]$ so that $a_0+\Sigma^*(A)$ is contained in $y(N)$-smooth numbers for some nonnegative integer $a_0$? When the function $y$ is bounded, Hegyv\'ari and S\'ark\"ozy~\cite[Theorem 4]{HS99} showed that the size of $A$ must be absolutely bounded, although they did not provide a quantitative bound. Here we focus on the case that $y$ is an increasing function in $N$.

For any finite $A\subseteq \N,$ define $P(A)$ to be the greatest prime factor of $\prod_{a\in A}a.$ Erd\H os, S{\'a}rk{\"o}zy, and Stewart~\cite{ESS94} made two conjectures about $P(\Sigma^*(A)),$ which we state here.
\begin{conj}\label{conj:ESS1}
    For $A\subseteq \N,$ $\lim_{|A|\rightarrow\infty} \frac{P(\Sigma^*(A))}{|A|} = \infty.$
\end{conj}
\begin{conj}\label{conj:ESS2}
    There is an absolute constant $c>0$ such that $P(\Sigma^*(A)) > c|A|^{2}.$
\end{conj}
In our setting, these two conjectures correspond to $|A| = o(y(N))$ and $|A| \ll y(N)^{1/2}$, which we are able to partially justify under some assumption on the growth of $y(N)$. More precisely, when $y(N)$ is large, our second framework applies and gives the following two corollaries.

\begin{cor}\label{cor:smooth1}
    Suppose $y(N)\ge (\log N)^2.$ If $a_0$ is a nonnegative integer and $A\subseteq [N]$ is a set such that $a_0+\Sigma^*(A)$ is contained in the set of $y(N)$-smooth numbers, then 
    \begin{equation}\label{eq:smooth}
        |A|\ll y(N)^{1/2}.
    \end{equation}
    Under the same setting, if $A\subseteq[N]$ is a multiset, then
    \begin{equation}\label{eq:smooth1}
        |A|\ll y(N).
    \end{equation}
\end{cor}

\begin{cor}\label{cor:smooth2}
    Suppose $\log N\leq y(N) \le (\log N)^{2-\eps}$ for some fixed $\eps > 0.$ If $a_0$ is a nonnegative integer and $A\subseteq [N]$ is a set such that $a_0+\Sigma^*(A)$ is contained in the set of $y(N)$-smooth numbers, then 
    \[|A| \ll_\eps \frac{\log N}{\log\log N}.\]
    Under the same setting, if $A\subseteq[N]$ is a multiset, then
    \begin{equation}\label{eq:smooth2}
        |A|\ll y(N).
    \end{equation}
\end{cor}
We remark that  estimate~\eqref{eq:smooth} in Corollary~\ref{cor:smooth1} is certainly optimal up to the implied constant, and was essentially established in \cite[Corollary 1]{ESS94}. For smaller $y(N)$, namely, $y(N)=(\log N)^{\alpha}$ with $\alpha \in (0,1]$, we have the following refinement of Corollary~\ref{cor:smooth2} based on our first framework.
\begin{thm}\label{thm:smooth1}
    Let $\alpha\in(0,1],K>0.$ There exist $\delta=\delta(\alpha)\in(0,1)$ and $C=C(\alpha,K)>0$ such that for any integer $0\leq a_0\leq N^K$ and any set $A\subseteq[N]$, if $a_0+\Sigma^*(A)$ is contained in $(\log N)^\alpha$-smooth numbers, then
    \[|A|\leq C(\log N)^{(1-\delta)\alpha}.\]
    Under the same setting, if $A\subseteq [N]$ is a multiset, then 
    \begin{equation}\label{eq:smooth3}
    |A|\leq C(\log N)^\alpha.
    \end{equation}
\end{thm}

We briefly discuss the sharpness of the results above. Our bounds~(\ref{eq:smooth1}), \eqref{eq:smooth2}, \eqref{eq:smooth3} are optimal up to the implied constant, as can be seen from the example where $a_0=0$ and $A$ is a multiset consisting of the element $1$ with multiplicity $y(N)$.

\subsection{Other applications: Hilbert cubes in primes, squarefree numbers}
In this section, we present several straightforward applications of our new frameworks, with complete proofs. 

For sets $S$ like primes or squarefree numbers, the question about the largest $A$ with $\Sigma^*(A)\subseteq S$ becomes trivial since a simple application of the pigeonhole principle would give $|A|\leq 3$. Hence, it is more meaningful to consider Hilbert cubes in primes and squarefree numbers.

\medskip

\textbf{Hilbert cubes in primes.} Additive patterns in primes have been studied extensively, with many conjectures remaining unsolved. We refer the reader to \cite{BRS14,EH15,GT08,TZ23} for some recent developments. Hegyv\'ari and S\'ark\"ozy~\cite{HS99} considered the question of Hilbert cubes contained in the set of primes. Theorems~\ref{thm:main4} and~\ref{thm:multi_incom} readily recover the best-known bound due to Woods \cite[Theorem 2]{W04} (independently by Elsholtz \cite{E03}) when $A$ is a set, and imply a conditionally sharp upper bound on $|A|$ when $A$ is a multiset. 

\begin{cor}
If $A \subseteq[N]$ is a multiset such that there is a nonnegative integer $a_0$ such that $a_0+\Sigma^*(A)$ is contained in the set of primes, then $|A|\ll \log N. $  If $A$ is assumed to be a set, then $|A|\ll \frac{\log N}{\log \log N}$.
\end{cor}

\begin{proof}
Let $A$ be a multiset $\{a_1,a_2,\ldots, a_d\}$. We may assume that $d$ is even without loss of generality. Let $a_0'=a_0+\sum_{j=d/2+1}^{d}a_j\geq d/2$ and let $A'=\{a_1,a_2,\ldots, a_{d/2}\}$. Note that $a_0'+\Sigma^*(A') \subset a_0+\Sigma^*(A)$. If $p<d/2$ is a prime, then the members of $\Sigma^*(A')$ do not run over all residue classes modulo $p$. Indeed, if $a_0'+\Sigma^*(A')$ contains an element $x$ such that $x$ is a multiple of $p$, then $x$ is not a prime since $x\geq d/2>p$. 

We may assume that $d\geq 12\log N$ for otherwise we are done. Let $B = \{p: p\leq 12\log N\}$. We have 
\[\sum_{p<d/2}\log p\geq \sum_{p<6\log N} \log p \geq 2\log N.\]
It follows from Theorem~\ref{thm:multi_incom} that
\[
d/2=|A'|\leq 2\max_{b\in B}b \leq 24\log N,
\]
that is, $d\ll \log N$.

Finally, assume that $a_1,a_2,\ldots, a_d$ are distinct. Then we have $a_0'>d^2/4$ and a similar argument as above shows that if $p\leq d^2/4$ is a prime, then the members of $\Sigma^*(A')$ do not run over all residue classes modulo $p$. We may assume $d\geq \log N/\log \log N$ for otherwise we are done. Then we can apply Theorem~\ref{thm:main4} with $B=\{p:p<2\log N\}$ to get the desired upper bound on $|A|$.
\end{proof}

For the lower bound construction, Woods \cite[Theorem 19]{W04} constructed a set 
$A = \{j \prod_{p\leq d^2} p: 1\leq j \leq d\} \subseteq [N]$
with $d\gg (\log N)^{1/2}$ such that for any prime $p$, $\Sigma^*(A)$ does not cover all residue classes modulo $p$. He then pointed out that assuming the linear case of Schinzel’s Hypothesis H \cite{SS58}, for such a set $A$, we can find a positive integer $a_0$ such that $a_0+\Sigma^*(A)$ is contained in the set of primes. If we allow repetition, then we can instead take $A$ to be the multiset consisting of a single element $\prod_{p\leq d} p$ with multiplicity $d \gg \log N$; again, assuming Schinzel’s Hypothesis H, we can see that there exists some $a_0>0$ so that $a_0+\Sigma^*(A)$ is contained in the set of primes. %The best-known unconditional lower bound comes from \cite[Theorem 3]{HS99}, where they proved the existence of a set $A$ with $|A|\gg \log\log N$ under the extra restriction that $a_0+\Sigma^*(A) \subseteq [N]$.

A slight modification of such a construction also works for primes in arithmetic progressions. For example, let $A$ be the multiset consisting of a single element $4\prod_{p\leq d}p$ with multiplicity $d\gg \log N$. Consider linear polynomials \[f_j(x) = 4x+1+4j\prod_{p\leq d}p, \qquad 1\leq j\leq d.\] Assuming Schinzel's Hypothesis H, there exists some positive integer $n$ such that $f_j(n)$'s are simultaneously primes congruent to $1\pmod 4$. This also provides a conditional lower bound construction for Hilbert cubes in $E = \{a^2+b^2:a,b\in\Z\}$. As for the upper bound, choosing $\mathcal{P}$ to be the set of primes congruent to $3\pmod 4$ in Theorem~\ref{thm:powerful} immediately implies the following corollary.
\begin{cor}
    If $A \subseteq[N]$ is a multiset such that there is a nonnegative integer $a_0$ with $a_0+\Sigma^*(A)\subseteq E$, then $|A|\ll \log N. $  If $A$ is assumed to be a set, then $|A|\ll \frac{\log N}{\log \log N}$.
\end{cor}
As discussed above, in the multiset setting, the upper bound is conditionally sharp.

\smallskip

\textbf{Hilbert cubes in squarefree numbers.} 
Sumsets contained in squarefree numbers were studied by Erd\H os and S\'ark\"ozy \cite{ES87}, G. S\'ark\"ozy \cite{S92}, Konyagin~\cite{K04}, and very recently van Doorn and Tao~\cite{vDT25}. In particular, Konyagin~\cite[Theorem 4]{K04} showed that the length of a longest arithmetic progression in the set of squarefree numbers up to $N$ is of the order $(\log N)^2$. As for Hilbert cubes, Bergelson and Ruzsa~\cite{BR02} used ergodic theory to show that the set of squarefree numbers contains an infinite Hilbert cube (they proved a stronger result about $\overline{\mathrm{IP}}$ set in translations of squarefree numbers). This motivates us to study Hilbert cubes contained in finite truncations of the set of squarefree numbers. In fact, Theorem~\ref{thm:multi_incom} implies the following corollary.
\begin{cor}
    If $A\subseteq[N]$ is a multiset and there exists a nonnegative integer $a_0$ such that $a_0+\Sigma^*(A)$ is contained in the set of squarefree numbers, then $|A|\ll (\log N)^2.$
\end{cor}
\begin{proof}
Let $B = \{p^2: p\leq 6\log N\}$. Since $a_0+\Sigma^*(A)$ is contained in the set of squarefree numbers, for any $b\in B,$ the members of $\Sigma^*(A)$ do not run over all residue classes modulo $b.$ Notice that 
$$
\sum_{b\in B}\Lambda(b) = \sum_{p\leq 6\log N}\log p > 2\log N,
$$
thus it follows from Theorem~\ref{thm:multi_incom} that
\[
|A|\leq 2\max_{b\in B}b \leq 72(\log N)^2.\qedhere
\]
\end{proof}

    The bound above is sharp up to the implied constant. One may wonder what bound we can get for a \emph{set} $A\subseteq[N]$. Unfortunately, our Theorem~\ref{thm:main4} only produces a bound of the same shape $O((\log N)^2)$ with a better implied constant.

\section{The first framework: arithmetic progressions and sumsets}\label{sec:1stframework}

In this section, we prove the following strengthening of Theorem~\ref{main} concerning Hilbert cubes of the form $a_0+\Sigma^*(A)$ with $a_0$ fixed. Note that in the following, the functions $f$ and $g$ depend on $a_0$. Indeed, for certain $a_0$, one can take advantage of the arithmetic property of $a_0$ to obtain stronger bounds on $f$ and $g$; we refer to the discussion on Hilbert cubes in perfect powers in Section~\ref{sec:perfectpower}.

\begin{thm}\label{main_strong}
Let $S$ be a subset of positive integers, $k\geq 2$ be a positive integer, and $a_0$ be a nonnegative integer. Let  $f,g$ be functions defined on $\N$ such that $f(N)\geq 2$ and the following two conditions are satisfied for all $N\in\N$:
\begin{enumerate}[(1)]
\item For any sets $B_1,B_2,\ldots,B_k\subseteq [N]$ satisfying $a_0+\sum_{i\in I}B_i\subseteq S$ for all nonempty subsets $I$ of $[k]$, we have \[\min_{1\leq i \leq k}|B_i|\leq f(N).\]

\item $S \cap [a_0, a_0+N]$ contains no arithmetic progression of length $g(N)$.
\end{enumerate}
Then we have the following:
\begin{enumerate}
    \item If $A \subseteq \N$ is a multiset with  $a_0+\Sigma^*(A)\subseteq S\cap [a_0,a_0+N]$, then we have
    \begin{equation}\label{eq:APcube}
|A|\ll k g(N)\log f(N),
\end{equation}
where the implied constant is absolute. 

\item Assume that $f(N) \to \infty$ as $N \to \infty$. For any $\eps>0,$ if $A\subseteq \N$ is a multiset with $a_0+\Sigma^*(A)\subseteq S\cap [a_0,a_0+N],$ then
    \begin{equation}\label{eq:APm}
    |A|\ll_{k,\eps}\max\{f(N)^\eps,g(N)\}.
    \end{equation}

\item Assume that $f(N) \to \infty$ as $N \to \infty$. For any $\eps>0,$ if $A\subseteq \N$ is a set with $a_0+\Sigma^*(A)\subseteq S\cap [a_0,a_0+N],$ we have 
    \begin{equation}\label{eq:APf}
    |A|\ll_{k,\eps}\max\{f(N)^\eps,g(N)^{1-\delta}\},
    \end{equation}
    where we can take $$\delta=\frac{1}{(\lfloor2(3k/\eps-1)\rfloor+1)!}.$$
\end{enumerate}
\end{thm}

Next, we briefly explain why part (b) is sharp for a random set almost surely.
\begin{rem}\label{rem:random}

Let $S$ be a random set formed by picking each positive integer independently with probability $1/2$. Next, we show that almost surely, $F_2(N)\leq G(N)^4$ whenever $N>N_0(S)$. It then follows from Theorem~\ref{main}(b) that almost surely, all Hilbert cubes in $S\cap[N]$ have dimension bounded above by $CG(N)$ for some absolute constant $C$, thereby confirming that Theorem~\ref{main}(b) is sharp up to the implied constant almost surely in this setting. 

By the Borel-Cantelli Lemma, to show the required estimate, it suffices to prove that
\begin{equation}\label{eq:BC}
\sum_{N=1}^\infty\mathbf{P}(F_2(N)>(\log N)^3)
<\infty\quad \text{and} \quad 
\sum_{N=1}^\infty\mathbf{P}(G(N)<c\log N)<\infty
\end{equation}
for some sufficiently small absolute constant $c>0$. To prove the estimate related to $F_2$, we use a result of Mrazovi\'c on the expansion property of random Cayley sum graphs over finite groups. From \cite[Section 3]{M16}, by viewing $S\cap[N]$ as a subset of $\Z_N$, when $N$ is sufficiently large, we have \[\mathbf{P}(F_2(N)>(\log N)^3)\ll N^2e^{-10\log N}=N^{-8}.\] The estimate related to $G$ in inequality~\eqref{eq:BC} is implicit in a paper of Erd\H os and R{\'e}nyi \cite{ER70}, and here we provide a short proof. If $G(N)<c\log N,$ then each subinterval of length $c\log N$ cannot be contained in $S$. Therefore by dividing $[N]$ into disjoint subintervals of length $c\log N$ and independence property, we have
$$
\mathbf{P}(G(N)<c\log N)\leq \left(1-\bigg(\frac{1}{2}\bigg)^{c\log N}\right)^{\lfloor N/c\log N \rfloor}\ll \exp\bigg(-\frac{N^{1-c\log2}}{c\log N}\bigg)\leq N^{-2}
$$
when $c=1/2$ and $N$ is large. This finishes the proof of the estimate~\eqref{eq:BC}, as required. 
\end{rem}

\subsection{Proof of Theorem~\ref{main_strong}(a)}
We begin with a proof of Theorem~\ref{main_strong}(a). We need the following cube lemma by Dietmann and Elshotlz \cite[Lemma 1.4]{DE15}; see also Schoen \cite[Section 2.2]{S11}.
\begin{lem}[Dietmann and Elshotlz]\label{lem:cube}
Let $t\geq 3$ be a positive integer, and let $S$ be a set of integers without an arithmetic progression of length $r$. Moreover, let $a_0$ be an integer and $a_1,\ldots, a_d$ be non-zero integers. If $H=H(a_0;a_1,a_2, \ldots, a_d)\subseteq S$, then $$|H|\geq 2(r/(r-1))^{d-1}-1.$$
\end{lem}

Now we are ready to prove Theorem~\ref{main_strong}(a).

\begin{proof}[Proof of Theorem~\ref{main_strong}(a)]
Let $A=\{a_1, a_2, \ldots, a_d\}\subset \N$ be a multiset with $a_0+\Sigma^*(A) \subseteq S \cap [a_0,a_0+N]$. We may assume that $d\geq 2k$, for otherwise we are done. For each $1\leq i \leq k$, set $$B_i=\Sigma^*(\{a_j: (i-1)\lfloor d/k\rfloor +1\leq j \leq i\lfloor d/k\rfloor\})\subseteq [N].$$ Then condition (1) implies that $\min_{1\leq i \leq k} |B_i|\leq f(N)$, say $|B_1|\leq f(N)$. Then we have $$H=H(a_0; a_1, a_2, \ldots, a_{\lfloor d/k\rfloor})\subseteq (S \cap [a_0,a_0+N]) \cup \{a_0\}.$$ By condition (2), $(S \cap [a_0,a_0+N]) \cup \{a_0\}$ contains no arithmetic progression of length $g(N)+1$. Thus, Lemma~\ref{lem:cube} implies that
$$
2((g(N)+1)/g(N))^{\lfloor d/k\rfloor-1}-1 \leq |H|\leq f(N)+1.
$$
It follows that $d/(kg(N))\ll \log f(N)$, that is, $d\ll kg(N)\log f(N)$.
\end{proof}

\begin{rem}
Let $m\geq 2$ and let $S_m$ be the set of $m$-th powers. Theorem~\ref{main_strong}(a) quickly implies that if $a_0$ is a non-negative integer and $A \subseteq [N]$ is a multiset with $a_0+\Sigma^*(A)\subseteq S_m$, then $|A|\ll_{m} \log \log N$, which recovers \cite[Theorem 1]{DE12} and \cite[Theorem 1.1]{DE15} by Dietmann and Elsholtz. Indeed, We can take $g(N)=4$ as mentioned in Section~\ref{sec:intro_perfectpower}; for $k=2$, a standard application of Gallagher's larger sieve shows that $f(N)\ll_{m} \log N$ (see for example Gyarmati \cite[Theorem 9]{G01}). 

On the other hand, it is a folklore conjecture that for each $m\geq 2$, the dimension of a Hilbert cube contained in $S_m$ is absolutely bounded \cite{BEF90, S07}. For example, the conjecture readily follows from the uniformity conjecture \cite{CHM}. 
\end{rem}

\subsection{Proof of Theorem~\ref{main_strong}(b) and (c)}
The proofs of Theorem~\ref{main_strong}(b) and (c) are much more involved. The following is an outline of the proof strategy:
\begin{enumerate}
    \item Assume that $|A|\geq f(N)^{\eps}$. Deduce that $|\Sigma^*(A)|=|A|^{O_{k,\eps}(1)}$ using condition (1). 
    \item Use an inverse Littlewood-Offord theorem to show that $A$ is ``close to" a generalized arithmetic progression. 
    \item Deduce that $\Sigma^*(A)$ contains an arithmetic progression of length $|A|^{1/(1-\delta)}$ (when $A$ is a set) or $O(|A|)$ (when $A$ is a multiset). Comparing this with condition (2), we obtain the desired bound in terms of $g(N)$.
\end{enumerate}

\medskip

We first recall some basic terminology on generalized arithmetic progressions. For a positive integer $m$, a \emph{generalized arithmetic progression over $\Z^m$} is of the form 
$$
P=\bigg\{\sum_{i=1}^{r} n_i v_i: M_i\leq n_i\leq N_i \text{ for all } 1\leq i \leq r\bigg \},
$$
where $v_1, \ldots, v_r \in \Z^m$, and $M_1, N_1, \ldots, N_r$ are integers. The \emph{rank} of $P$ is $r$ and the \emph{volume} of $P$ is $\operatorname{vol}(P)=\prod_{i=1}^{r} (N_i-M_i+1)$. Moreover, we say $P$ is \emph{proper} if $|P|=\operatorname{vol}(P)$, and $P$ is \emph{symmetric} if $M_i=-N_i$ for all $1\leq i \leq r$.

Another important tool in our proof is the Freiman isomorphism, whose definition is stated below.
\begin{defn}
    Let $s$ be a positive integer, and $A,B$ be subsets of some abelian groups. We say a function $\psi:A\rightarrow B$ is a \emph{Freiman $s$-homomorphism} if for any $x_1,\ldots,x_s,y_1,\ldots,y_s\in A,$ if \[x_1+\cdots+x_s = y_1+\cdots+y_s,\] then \[\psi(x_1)+\cdots+\psi(x_s) = \psi(y_1)+\cdots+\psi(y_s).\]
    If $\psi$ is bijective with $\psi^{-1}$ also being a Freiman $s$-homomorphism, then we say $\psi$ is a \emph{Freiman $s$-isomorphism}.
\end{defn}

As preparation, we need several structural results in additive combinatorics. The first ingredient is an optimal inverse Littlewood-Offord theorem due to Nguyen and Vu~\cite[Theorem 2.1]{NV11}.
\begin{thm}[Nguyen and Vu]\label{thm:NV}
  Let $\eta_i, i=1,\dots, n$ be iid Bernoulli random variables, taking values $\pm 1$ with probability $\frac{1}{2}$. Given
a multiset $V$ of $n$ integers  $v_1, \dots, v_n$,  define the \emph{concentration probability} as
$$\rho(V) := \sup_{x} \operatorname{Pr}( v_1 \eta_1+ \dots
v_n \eta_n=x). $$
Let $C$ and $1> \eps$ be positive constants. There is a constant $c_1= c_1(\eps, C)$ such that the following holds: if $\rho (V)  \ge  n^{-C}$, then there exists a proper symmetric generalized arithmetic progression $Q$ of rank $r= O_{C, \eps} (1)$ which contains all but at most $\eps n$
 elements of $V$ (counting multiplicity), where
 $$|Q| = O_{C, \eps} ( \rho (V)^{-1} n^{- \frac{r}{2}}). $$
\end{thm}

\begin{cor}\label{cor:NV}
    Let $C>0$ be a constant, $A=\{a_1,\ldots,a_n\}$ be a multiset of integers satisfying $|\Sigma(A)| \le n^C.$ Then there is a proper symmetric generalized arithmetic progression $Q$ of rank $r\leq O_C(1)$ that contains at least half of the elements in $A$ (counting multiplicity), where \[|Q| = O_C(n^{C-r/2}).\]
\end{cor}
\begin{proof}
For each subset $I$ of $[n]$, let $\eta_i=1$ if $i\in I$, and $\eta_i=-1$ otherwise; observe that
$$
\sum_{i\in I} a_i=\sum_{i=1}^{n} \frac{\eta_i+1}{2} a_i=\frac{1}{2} \sum_{i=1}^n \eta_i a_i +\frac{1}{2}\sum_{i=1}^n a_i.
$$
The observation implies that
$$
\rho(A)=\sup_{x} \operatorname{Pr}\bigg(\sum_{i\in I}a_i=x\bigg),
$$
where $I$ is a random subset of $[n]$. Since $|\Sigma(A)|\leq n^C$, by pigeonhole principle, it follows that $\rho(A) \geq n^{-C}$. The corollary then follows from Theorem~\ref{thm:NV}.
\end{proof}

The second result is due to Szemer{\'e}di and Vu~\cite[Theorem 7.1]{SV06}.
\begin{thm}[Szemer{\'e}di and Vu]\label{thm:SV}
    For any fixed positive integer $d$, there are positive constants $C$ and $c$ depending on $d$ such that the following holds: if $n, \ell$ are positive integers and a set $A\subseteq[n]$  satisfy $\ell\leq |A|/2$ and $\ell^d|A|\geq Cn,$ then $$\ell^*A=\bigg\{\sum_{i=1}^{\ell} a_i: a_1, a_2, \ldots, a_{\ell} \in A \text{ are distinct}\bigg\}$$ contains a proper generalized arithmetic progression of rank $d'$ and volume at least $c\ell^{d'}|A|$, for some integer $1\leq d'\leq d.$
\end{thm}
As a corollary, we can find a long arithmetic progression in $\Sigma^*(A)$ as long as $|A|$ is not too small.
\begin{cor}\label{cor:AP}
    For any fixed positive integer $d\geq 2$, there are positive constants $C$ and $c$ depending on $d$ such that the following holds: if $n$ is a positive integer and a set $A\subseteq[n]$ satisfy $|A|^d\geq Cn,$ then $\Sigma^*(A)$ contains an arithmetic progression of length at least $c|A|^{1+1/(d-1)}.$ 
\end{cor}
\begin{proof}
Without loss of generality, we may assume that $|A|$ is even. Let $\ell=|A|/2$. Let $C_{d-1}$ and $c_{d-1}$ be the constants from Theorem~\ref{thm:SV} for $d-1$. If $|A|^d\geq 2^{d-1}C_{d-1}n$, then $\ell^{d-1} |A|=|A|^d/ 2^{d-1}\geq C_{d-1}n$. 
Since $\ell^* A \subseteq \Sigma^*(A)$,   Theorem~\ref{thm:SV} implies that $\Sigma^*(A)$ contains a proper generalized arithmetic progression of rank $d'$ and volume at least $c_{d-1}\ell^{d'}|A|$ for some integer $1\leq d'\leq d-1.$ It follows that $\Sigma^*(A)$ contains an arithmetic progression of length at least 
$$
(c_{d-1}\ell^{d'}|A|)^{1/d'}=c_{d-1}^{1/d'} \ell |A|^{1/d'}\geq \frac{c_{d-1}^{1/d}}{2} |A|^{1+1/(d-1)},
$$
as required.
\end{proof}
The next lemma serves as a high-dimensional version of Corollary~\ref{cor:AP}. 
\begin{lem}\label{lem:AP}
    For any fixed positive integers $r\geq 2$ and $d,$ there are positive constants $C$ and $c$ depending on $d$ and $r$ such that the following holds: for any positive integers $M_1,M_2,\ldots,M_d$ and any set $A\subseteq[M_1]\times[M_2]\times\cdots\times[M_d]$ satisfying $|A|\geq C(M_1M_2\cdots M_d)^{(d+1)/r}$, $\Sigma^*(A)$ contains an arithmetic progression of length $c|A|^{1+1/(r-1)}.$
\end{lem}
\begin{proof}
    The strategy is to construct a Freiman 2-isomorphism from a larger cube to a subset of $\Z,$ and then apply Corollary~\ref{cor:AP} to locate a long arithmetic progression. 
    
    Let $L=M_1M_2\cdots M_d.$ We define a sequence $a_n$ by taking $a_1=1$ and $a_n=2L\sum_{i=1}^{n-1}a_iM_i+1$ for all $2\leq n\leq d.$ Let $B=[M_1L]\times[M_2L]\times\cdots\times[M_dL]$.    
    Now we can construct the map $\psi:B\rightarrow\Z$ as 
    \[\psi(x_1,x_2,\ldots,x_d) = \sum_{i=1}^da_ix_i.\]
    It is easy to check $\psi(B)$ is a subset of $[4^d(M_1M_2\cdots M_d)^{d+1}]$ and $\psi:B \to \psi(B)$ is a Freiman 2-isomorphism. Since $A\subseteq[M_1]\times[M_2]\times\cdots\times[M_d],$ we have $|\psi(A)|=|A|.$ It follows from Corollary~\ref{cor:AP} that there exist positive constants $C$ and $c$ such that if $|\psi(A)|\geq C(M_1M_2\cdots M_d)^{\frac{d+1}{r}},$ then $\Sigma^*(\psi(A))$ contains an arithmetic progression $P$ of length $c|\psi(A)|^{1+1/(r-1)}.$ Notice that $\Sigma^*(A)\subseteq[M_1L]\times[M_2L]\times\cdots\times[M_dL]=B,$ thus $\Sigma^*(\psi(A))=\psi(\Sigma^*(A))$ and $\Sigma^*(A)$ contains an arithmetic progression $\psi^{-1}(P)$, which has the same length as $P.$ 
\end{proof}
We also need a multiset version of Lemma~\ref{lem:AP}. To achieve that, we need the following result of Szemer{\'e}di and Vu~\cite[Corollary 5.2]{SV06}.
\begin{thm}[Szemer{\'e}di and Vu]\label{thm:SV_multi}
    For any fixed positive integer $d$, there are positive constants $C$ and $c$ depending on $d$ such that the following holds: if $n, \ell$ are positive integers and $A_1,\ldots,A_{\ell}$ are subsets of $[n]$ of size $m$ such that $\ell^dm\geq Cn$, then $A_1+\cdots+A_\ell$ contains an arithmetic progression of length $c\ell m^{1/d}.$
\end{thm}
\begin{cor}\label{cor:APm}
    For any fixed positive integer $d$, there are positive constants $C$ and $c$ depending on $d$ such that the following holds: if $A\subseteq[n]$ is a multiset satisfying $|A|^d\geq Cn,$ then $\Sigma^*(A)$ contains an arithmetic progression of length at least $c|A|.$
\end{cor}
\begin{proof}
Without loss of generality, we may assume that $|A|$ is even. Let $\ell=|A|/2$. If the multiset $A$ contains an element $a$ of multiplicity at least $\ell$, then the arithmetic progression $\{a,2a,\ldots,\ell a\}$ is contained in $\Sigma^*(A)$ and we are done. In the other case, we can partition $A$ into $\ell$ sets $A_1,\ldots,A_\ell$ such that each $A_i$ consists of exactly $2$ distinct elements. Now the sumset $A_1+\cdots+A_\ell$ is a subset of $\Sigma^*(A)$. Since $2\ell^d>(|A|/2)^d=2^{-d}|A|^d,$ by choosing $C$ and $c$ accordingly based on Theorem~\ref{thm:SV_multi}, we conclude that $\Sigma^*(A)$ contains an arithmetic progression of length $c|A|.$
\end{proof}

Similar to the proof of Lemma~\ref{lem:AP}, we have the following high-dimensional version of Corollary~\ref{cor:APm}.
\begin{lem}\label{lem:APm}
    For any fixed positive integers $r$ and $d,$ there are positive constants $C$ and $c$ depending on $d$ and $r$ such that the following holds: for any positive integers $M_1,M_2,\ldots,M_d$ and any multiset $A\subseteq[M_1]\times[M_2]\times\cdots\times[M_d]$ satisfying $|A|\geq C(M_1M_2\cdots M_d)^{(d+1)/r}$, $\Sigma^*(A)$ contains an arithmetic progression of length $c|A|.$
\end{lem}
Finally, we need the following lemma on proper symmetric generalized arithmetic progressions.
\begin{lem}[{\cite[Theorem 3.40]{TV09}}]\label{lem:GAP}
    Let $P$ be a symmetric generalized arithmetic progression of rank $r$ in $\Z.$ If $r\geq 2$ and $P$ is not proper, then $P$ is contained in a proper symmetric generalized arithmetic progression $Q$ of rank at most $r-1$ and $|Q|\leq r^{Cr^3}|P|$, where $C$ is an absolute constant.
\end{lem}

Now we are ready to prove Theorem~\ref{main_strong}(b) and (c).

\begin{proof}[Proof of Theorem~\ref{main_strong}(b) and (c)]
Let $A \subseteq \N$ be a multiset with $a_0+\Sigma(A)\subseteq S\cap [a_0,a_0+N]$. We may assume that $|A|\geq f(N)^\eps$, for otherwise we are done. In particular, we can assume $A$ is sufficiently large by the assumption $f(N) \to \infty$ as $N\to \infty$.
    
    We claim that there exists a partition of $A$ into $k$ submultisets $A_1,A_2,\ldots,A_k$ such that  
    \[\min_{1\leq i\leq k}|\Sigma(A_i)|\geq 2^{1/k-1}|\Sigma(A)|^{1/k}.\]
    To see this, pick $A_1\subseteq A$ to be minimal with the property that $|\Sigma(A_1)| \geq 2^{1/k-1}|\Sigma(A)|^{1/k},$ and inductively for $2\leq i\leq k-1,$ pick $A_{i}\subseteq A\setminus(\cup_{j=1}^{i-1} A_j)$ to be minimal with the property that $|\Sigma(A_i)| \ge 2^{1/k-1}|\Sigma(A)|^{1/k}$. Observe that for each $1\leq i \leq k-1$, we have
    $$
    2^{1/k-1}|\Sigma(A)|^{1/k} \leq |\Sigma(A_i)| \leq 2^{1/k}|\Sigma(A)|^{1/k}.
    $$
    Let $A_k=A\setminus(\cup_{i=1}^{k-1}A_k).$ Since $$|\Sigma(A)| = |\Sigma(A_1) +\cdots+ \Sigma(A_k)| \leq |\Sigma(A_1)|\cdots|\Sigma(A_k)|,$$ it follows that $|\Sigma(A_k)| \ge 2^{1/k-1}|\Sigma(A)|^{1/k},$ which finishes the proof of the claim.
    
    Since $a_0+\Sigma(A)\subseteq S\cap [a_0,a_0+N]$, we have $\Sigma^*(A_i)\subseteq[N]$ for all $i\in[k]$ and $$a_0+\sum_{i\in I}(\Sigma^*(A_i))\subseteq a_0+\Sigma^*(A)\subseteq S$$ for all nonempty $I\subseteq[k].$ Thus it follows from condition (1) that $\min_i|\Sigma(A_i)|\leq f(N)+1.$ Hence
    \[|\Sigma(A)|\leq 2^{k-1}\min_i|\Sigma(A_i)|^k\leq 2^{k-1}(f(N)+1)^k< 4^{k}|A|^{k/\eps}\leq |A|^{3k/\eps},\]
    where we used the assumption that $|A|^{1/\eps}\geq f(N)\geq 2$. 

    Next, we first assume that $A$ is a set and focus on proving (c). It follows from Corollary~\ref{cor:NV} that there is a proper symmetric generalized arithmetic progression $Q$ of rank $r$ such that $|Q\cap A|\geq |A|/2$ and $|Q|\ll_{k,\eps}|A|^{3k/\eps-r/2}.$ In particular we have $r\leq 2(3k/\eps-1).$ Since $Q$ is proper and symmetric, we can write $Q=\{\sum_{i=1}^r x_i v_i: |x_i|\leq M_i\}$ for some positive integers $M_1, M_2, \ldots, M_r$ and some nonzero integers $v_1,v_2, \ldots, v_r$. We can partition $Q$ into $3^r$ sets according to the sign pattern of $(x_1,x_2, \ldots, x_r)$. In particular,  
    by pigeonhole and relabeling of indices, there is an integer $1\leq r'\leq r$, such that the proper generalized arithmetic progression
    $Q'=\{\sum_{i=1}^{r'} x_i w_i: x_i \in [M_i]\}$ satisfies that $|Q' \cap A|\geq |Q\cap A|/3^r$, where $w_i\in \{v_i, -v_i\}$ for each $1\leq i \leq r'$. Let $A'=Q' \cap A$; then we have $$A' \subseteq Q', \quad |A'|\geq |A|/(2 \cdot 3^r)\gg_{k,\eps} |A|, \quad \text{and} \quad |Q'|\ll_{k, \eps} |A'|^{3k/\eps-r/2}.$$ Let $\psi:[M_1]\times[M_2]\times\cdots\times[M_{r'}]\rightarrow Q'$ be the Freiman 2-homomorphism such that 
    $$\psi(x_1,x_2, \ldots, x_{r'})=\sum_{i=1}^{r'} x_i w_i.$$ Since $Q'$ is proper, $\psi$ is bijective. Thus, 
 $$|\psi^{-1}(A')|=|A'|\gg_{k,\eps} |Q'|^{\frac{1}{3k/\eps-r/2}}=(M_1M_2\cdots M_{r'})^{\frac{1}{3k/\eps-r/2}}.$$ 

 Given an integer $s\geq 2$ with
    \[\frac{1}{3k/\eps-r/2}>\frac{r+1}{s}\geq \frac{r'+1}{s},\]
    it follows from Lemma~\ref{lem:AP} that $\Sigma^*(\psi^{-1}(A'))$ contains an arithmetic progression $P$ of length $\gg_{k,\eps,s} |A|^{1+1/(s-1)}.$  Since $\psi$ can be extended to a group homomorphism from $\Z^{r'}$ to $\Z$ in a natural way, we have $$\psi(\Sigma^*(\psi^{-1}(A')))=\Sigma^*(\psi(\psi^{-1}(A')))=\Sigma^*(A').$$ Hence $\Sigma^*(A')$ contains the arithmetic progression $\psi(P).$ Let $(t_1,t_2,\ldots,t_{r'})$ be the step of $P.$  Note that $|t_i|\leq M_iL,$ where $L=M_1M_2\cdots M_{r'}.$ Next we consider two cases.
    
    If $\psi(t_1,t_2,\ldots,t_{r'})\neq 0,$ then $$|a_0+\psi(P)|=|P|\gg_{k,\eps,s} |A|^{1+1/(s-1)}.$$ Since $a_0+\psi(P)\subseteq a_0+\Sigma(A)\subseteq S\cap[a_0,a_0+N],$ we know from condition (2) that $|a_0+\psi(P)|\leq g(N),$ which implies $|A|\ll_{k,\eps,s} g(N)^{\frac{s-1}{s}}.$ 

    If $\psi(t_1,t_2,\ldots,t_{r'})=0,$ that is,  
    \[t_1w_1+t_2w_2+\cdots+t_{r'}w_{r'}=0.\]  Then the symmetric generalized arithmetic progression $$Q''=\{y_1w_1+y_2w_2+\cdots+y_{r'}w_{r
    '}:|y_i|\leq M_iL\}$$ is not proper. It follows from Lemma~\ref{lem:GAP} that $Q''$ is contained in a proper symmetric generalized arithmetic progression $R$ of rank at most $r'-1$ and $$|R|\ll_{r}|Q''|\ll_r L^{r+1}\ll_{k,\eps} |A|^{(r+1)(3k/\eps-r/2)}.$$ Since $A'\subseteq Q'\subseteq R,$ we can replace $Q$ by $R$ and repeat the above argument. This procedure must terminate because a (nontrivial) arithmetic progression is always proper. In the worst case, we end up with 
    \[|A|\ll_{k,\eps,s} g(N)^{\frac{s-1}{s}}\]
    where $s\in\N$ satisfies 
    \[\frac{1}{(3k/\eps-r/2)\cdot(r+1)!}>\frac{1}{s},\]
    where we recall that $r\leq 2(3k/\eps-1).$ The left-hand side of the above inequality is minimal when $r$ is maximal. Thus, we can choose $s=(\lfloor2(3k/\eps-1)\rfloor+1)!$ and $\delta=1/s$ to complete the proof.

Finally, we briefly explain how to prove (b). Since Corollary~\ref{cor:NV} also applies to multisets, the proof of (b) is very similar to the proof of (c) above, except we use Lemma~\ref{lem:APm} in place of Lemma~\ref{lem:AP}. We also need to come up with appropriate multiset analogues of the sets and set relations appearing in the above proof. For example, we need to view $Q' \cap A$ as the multiset $\{a: a \in A', a \in Q\}$, view $\psi^{-1}(A')$ as the multiset
$\psi^{-1}(A')=\{\psi^{-1}(a): a \in A'\}$, and interpret $A'\subseteq Q'$ as $a\in Q'$ for each $a\in A'$.
\end{proof}

\begin{rem}\label{rem:B_i}
Note that in the above proof of Theorem~\ref{main_strong}, we actually only used condition (1) for sets $B_i$ of the form $\Sigma^*(A_i)$, where $A_i \subseteq A$. Thus, if we know in advance that all the elements in $A$ are multiples of a fixed prime $p$, then to get the same conclusion of Theorem~\ref{main_strong}, it suffices to assume that condition (1) holds for $B_i$ consisting of multiples of $p$. This observation will be useful for our applications to Hilbert cubes in perfect powers in Section~\ref{sec:perfectpower}.
\end{rem}

As a by-product of the proof, we have the following inverse theorem for maximal Hilbert cubes in $S\cap[N]$.
\begin{thm}\label{thm:inverse}
    Let $A\subseteq\N$ be a maximal multiset so that $a_0+\Sigma^*(A)\subseteq S\cap[N]$ for some nonnegative integer $a_0$. Suppose for some fixed $k$ and $l\geq 1$, we have $f(N)\ll g(N)^l$ and $\lim_{N\rightarrow\infty}f(N)=\infty$, where $f$ and $g$ are defined in Theorem~\ref{main_strong}. Then, when $N$ is sufficiently large, there is a proper symmetric generalized arithmetic progression $Q$ of rank $r\leq 2(kl-1)$ such that $|Q\cap A| \ge |A|/2$ and $|Q| \ll |A|^{kl-r/2+0.2}.$ 
\end{thm}
\begin{proof}
    Since $A$ is maximal, we must have $|A|\geq g(N)$ due to the arithmetic progression construction. From the assumption $f(N)\ll g(N)^l$, when $N$ is sufficiently large, we have
    \[|A|\geq g(N)\geq (f(N)+1)^{k/(kl+0.1)}.\]
    It follows from the argument in the proof above that we can take $\eps=k/(kl+0.1)$ and obtain
    \[|\Sigma(A)|\leq 2^{k-1}(f(N)+1)^k\leq |A|^{kl+0.2}\]
    for sufficiently large $N.$ Now Corollary~\ref{cor:NV} implies that there is a proper symmetric generalized arithmetic progression $Q$ of rank $r$ such that $|Q\cap A|\geq|A|/2$ and $|Q|\ll |A|^{kl-r/2+0.2}$. In particular, we must have $r\leq 2(kl-1).$
\end{proof}
When $A$ is a maximal set with $a_0+\Sigma^*(A)\subseteq S\cap[N],$ we also have an inverse theorem with a similar conclusion, except that the bound on the rank of $Q$ becomes $r\leq 2(2kl-1).$

\section{Hilbert cubes in powerful numbers}\label{sec:powerful}

In this section, we prove the results stated in Section~\ref{sec:intro_powerful}.

\subsection{A local version}
We aim to study Hilbert cubes in powerful numbers. Note that if $n$ is a powerful number, then for each prime $p$, either we have $p\nmid n$ or $p^2\mid n$, so $n \notin \{p,2p, \ldots, (p-1)p\} \pmod {p^2}$. This motivates the setting of the following local version.
\begin{prop}\label{prop:local}
There are positive constants $K,K'$ such that the following holds: for any prime $p$ and any $a_0\in \Z_{p^2}$, if $A$ is a multiset in $\Z_{p^2}\setminus \{0,p,2p,\ldots, (p-1)p\}$ so that $a_0 + \Sigma^*(A)$ is disjoint from $\{p,2p,\ldots,(p-1)p\},$ 
then 
\begin{enumerate}[(1)]
    \item $\alpha(m)\le K\sqrt{p/m}$ for all $1\leq m\leq 2p$;
    \item $\alpha(2p+1)=0$;
    \item $|A|\leq K'p$;
    \item If $|A|\geq p^{0.51}$, then $\sum_{k=1}^{2p} k\alpha(k)\gg |A|^3/p$, where the implicit constant is absolute.
\end{enumerate}
Here $\alpha(m)$ denotes the number of elements in $\Z_{p}$ that appears at least $m$ times in $A$.
\end{prop}

\begin{rem}
     Our bound $\alpha(m)\ll \sqrt{p/m}$ is pointwise optimal up to the constant, and 
    $\sum_{k} k\alpha(k) \gg |A|^3/p$ is also optimal up to the constant. To see this, for fixed $1 \leq m \leq p,$ let $a_0=0$ and $A = \{1,2,\ldots,\lfloor \sqrt{p/m} \rfloor\}$ where each element has multiplicity $m.$ Then $a_0+\Sigma(A)$ is disjoint from $\{p,2p,\ldots,(p-1)p\}$ and $\alpha(m) = \lfloor \sqrt{p/m} \rfloor.$ Moreover, in this case,
    \[|A| = m \lfloor \sqrt{p/m} \rfloor,\ \sum_{k=2}^m k\alpha(k) \asymp m^2 \lfloor \sqrt{p/m} \rfloor \asymp |A|^3/p. \]
    The bound in Proposition~\ref{prop:local}(3) is also optimal up to the constant $K.$ For example, we can take $a_0=0$ and $A$ to be the multiset which contains a single element $1$ with multiplicity $p-1.$ As another example, let $a_0=0$ and $A=\{1, 2+p, 3+p, \ldots, \lfloor p^{1/2}/2 \rfloor + p\}$ where $1$ has multiplicity $\lfloor p/4 \rfloor$ and all the other elements have multiplicity exactly $1$. 
\end{rem}

A key ingredient in the proof of Proposition~\ref{prop:local} is the expansion property of subset sums in finite abelian groups. Such property has been studied extensively and we refer to a nice survey \cite{GG06}. In particular, we will apply the following lemma due to DeVos, Goddyn, Mohar, \v S\'amal~\cite[Theorem 1.6]{DGMS07}. Recall that for a subset $S$ of a finite abelian group $G$, the \emph{stabilizer} of $S$ is $\operatorname{stab}(S)=\{g\in G: g+S=S\}$; note $\operatorname{stab}(S)$ is a subgroup of $G$. 

\begin{lem}[DeVos, Goddyn, Mohar, \v S\'amal]\label{lem:DGMS}
Let $G$ be a finite abelian group. If $A\subseteq G \setminus \{0\}$ is a multiset and $H=\operatorname{stab}(\Sigma(A))$, then 
$$
|\Sigma^*(A)|\geq |H|+\frac{|H|}{64} \cdot \sum_{j \in \N}\rho_j^2,
$$
where $\rho_j$ is the number of nontrivial $H$-cosets of $G$ which contain at least $j$ elements of $A$.
\end{lem}

\subsection{Proof of Proposition~\ref{prop:local}}\label{subsec:local}\

(1) Fix $1\leq m \leq 2p$. Let $a_1,\ldots,a_d\in \Z_p$ be all the distinct residues that appear in $A\pmod p$. We partition $A$ into multisets $B_1, B_2, \ldots, B_d$ such that all elements in $B_i$ are congruent to $a_i$ modulo $p.$ Note that those $B_i$'s with size less than $m$ do not contribute to $\alpha(m)$, and we may simply drop those $B_i$'s. Thus, we may also assume that $|B_i| \geq m$ for $1\leq i \leq d$. 

For $1\leq i \leq d$, let $\Delta(B_i)$ be the number of distinct elements in $B_i$. Let \[I := \{1\leq i\leq d : \Delta(B_i)\geq 2\}.\]
For each $i\in I,$ by the pigeonhole principle, we can take a submultiset $C_i$ of $B_i$ such that $|B_i|/2\leq |C_i| < |B_i|$ and $C_i\cap (B_i \setminus C_i) = \emptyset.$

First, we show that $|I| < (128p/m)^{1/2}.$ Suppose otherwise that $|I|\geq (128p/m)^{1/2}$. We claim that $a_0+\Sigma^*(\bigcup_{i\in I} C_i)$ contains an element that is $0 \pmod p$. Indeed, if this is not the case, then $a_0+\Sigma^*(\bigcup_{i\in I} C_i)$ has trivial stabilizer in $\Z_p$; thus, applying Lemma~\ref{lem:DGMS} with $G=\Z_p$ and $H=\{0\}$ gives that $\Sigma^*(\bigcup_{i\in I} C_i)$ runs over at least $\frac{1}{64}|I|^2\cdot\frac{m}{2}\geq p$ distinct residue classes modulo $p$, a contradiction. Also, note that none of the sums in $a_0+\Sigma^*(\bigcup_{i\in I} C_i)$ is $0 \pmod p$ but not $0 \pmod {p^2}$. In particular, there is a sum in $a_0+\Sigma^*(\bigcup_{i\in I} C_i)$ that is $0 \pmod {p^2}$; say one of the element in the summand (other than $a_0$) is in $C_i$ for some $i\in I$, then we can just replace that element by an element in $B_i \setminus C_i$ to get a sum that is $0 \pmod {p}$ but not $0 \pmod {p^2}$, which again gives a contradiction. Therefore we have $|I| < (128p/m)^{1/2}.$

Next, we estimate the size of $[d]\setminus I.$ For any $i\in [d]\setminus I,$ $B_i$ contains a unique residue class modulo $p^2,$ say $b_i,$ with multiplicity at least $m.$ 
Suppose $|[d]\setminus I|\geq 2\cdot(64p/m)^{1/2}+2,$ pick $I'\subseteq [d]\setminus I$ with $|I'| = \lceil(64p/m)^{1/2}\rceil,$ by Lemma~\ref{lem:DGMS} with $G=\Z_{p^2}$, we have
\[\bigg|\Sigma^*\bigg(\bigcup_{i\in I'} B_i\bigg)\bigg|\geq \min\{p^2,1+\frac{m}{64}|I'|^2\}>p.\]
Thus, by the pigeonhole principle, we can find two distinct nonempty submultisets $D_1$ and $D_2$ of $\cup_{i\in I'}B_i$ such that 
\[\sum_{x \in D_1} x \not \equiv \sum_{x \in D_2} x \pmod {p^2}, \quad \sum_{x \in D_1} x \equiv \sum_{x \in D_2} x \pmod {p}.\]
However, since $|[d]\setminus(I\cup I')|>(64p/m)^{1/2}$, similar to the argument in the previous paragraph, we know $\Sigma^*(\cup_{i\notin I\cup I'} B_i)$ runs over all residue classes modulo $p.$ In particular, there is a nonempty submultiset $D_3$ of $\cup_{i\notin I\cup I'} B_i$ such that $\sum_{x\in D_3} x \equiv -a_0-\sum_{x\in D_1} x\pmod p$. It follows that $D_1\cup D_3$ and $D_2\cup D_3$ are different submultisets of $\cup_{i\notin I} B_i$ with the property that 
\[\sum_{x \in D_1 \cup D_3} x \not \equiv \sum_{x \in D_2 \cup D_3} x \pmod {p^2}, \quad \sum_{x \in D_1 \cup D_3} x \equiv \sum_{x \in D_2 \cup D_3} x \equiv -a_0 \pmod {p}.\]
In particular, $$\big(a_0+\Sigma^*(\cup_{i\notin I} B_i)\big) \cap \{p,2p,\ldots, (p-1)p\}\ne \emptyset,$$ a contradiction. Hence we have $|[d]\setminus I| \leq 2\cdot(64p/m)^{1/2}+2$. Combining the above estimates we obtain $\alpha(m)\ll (p/m)^{1/2}$ since $m\leq 2p$.

(2) Suppose otherwise that there is an integer $1\leq \theta \leq p-1$ and integers $b_1, b_2, \ldots, b_{2p+1}$ with $0\leq b_i\leq p-1$, such that the multiset $\{b_1p+\theta, b_2p+\theta,\ldots, b_{2p+1}p+\theta\}$ is contained in $A$. Choose an integer $1\leq \ell \leq p$ such that $\theta \ell \equiv -a_0 \pmod p$.

We claim that $b_1,\ldots, b_{2p+1}$ are all the same. Suppose otherwise that they are not; by relabeling, we may assume $b_{\ell} \neq b_{\ell+1}$. Observe that
$$
a_0+\sum_{i=1}^{\ell} (b_ip+\theta)\equiv a_0+\sum_{i=1}^{\ell-1} (b_ip+\theta) \ +(b_{\ell+1} p+\theta) \equiv 0\pmod {p}.
$$
Since \[\sum_{i=1}^{\ell} b_i \not \equiv \sum_{i=1}^{\ell-1} b_i \ +b_{\ell+1} \pmod {p},\]
it follows that
\begin{equation}\label{eq:modp^2}
a_0+\sum_{i=1}^{\ell} (b_ip+\theta)\not \equiv a_0+\sum_{i=1}^{\ell-1} (b_ip+\theta) +(b_{\ell+1} p+\theta) \pmod {p^2}.
\end{equation}
Note that both the sum in the left-hand side or the right-hand side of equation~\eqref{eq:modp^2} are in 
$a_0+\Sigma^*(A)$, while one of them in $\{p,2p,\ldots, (p-1)p\}$, violating the assumption.

Now consider
$$
a_0+\sum_{i=1}^{\ell} (b_ip+\theta)=a_0+\ell \theta +\ell b_1p, \quad a_0+\sum_{i=1}^{\ell+p} (b_ip+\theta)=a_0+(\ell+p) \theta +(\ell+p) b_1p.
$$
Note that 
$$
a_0+\sum_{i=1}^{\ell} (b_ip+\theta)\equiv a_0+\sum_{i=1}^{\ell+p} (b_ip+\theta)\equiv 0 \pmod p,
$$
while 
$$
a_0+\sum_{i=1}^{\ell} (b_ip+\theta)\not \equiv a_0+\sum_{i=1}^{\ell+p} (b_ip+\theta) \pmod {p^2}, 
$$
again a contradiction by the same reason as above. 

(3) By (1)(2), we have 
$$
|A|=\sum_{m=1}^{2p} \alpha(m)\leq K\sum_{m=1}^{2p} \sqrt{p/m}\ll p.
$$

(4) By (1) and (2), we have $\alpha(k)\leq K\sqrt{p/k}$ for all positive integer $k$. Since $|A|\geq p^{0.51}$, by the rearrangement inequality, we have 
$$\sum_{k\geq 2} k\alpha(k)\gg \sum_{k\leq T} \sqrt{pk}\gg \sqrt{p} T^{3/2},$$
where $T$ is the largest integer such that $\sum_{k \leq T} (p/k)^{1/2}\leq |A|$. We have $T\gg |A|^2/p$ and thus $\sum_{k} k\alpha(k)\gg |A|^3/p$.

\subsection{Proof of Theorem~\ref{thm:powerful}}
\iffalse
We prove the following strengthening of Theorem~\ref{thm:powerful}.
\begin{thm}\label{thm:spowerful}
    If $a_0$ is a non-negative integer and $a_1, a_2, \ldots, a_d \in [N]$ are positive integers such that $H(a_0; a_1, a_2, \ldots, a_d)$ is contained in the set $W_N,$ where 
    \[W_N = \{n\in\N : p \mid n \Rightarrow p^2 \mid n\ \text{for any prime }p \leq 4\log N\},\]
    then $d\ll \log N.$ 
\end{thm}
\begin{proof}
    Given $a_1, a_2, \ldots, a_d \in [N],$ we claim that there is a prime $p\leq 4\log N$ so that $p$ does not divide at least half of the $a_i$'s. For the sake of contradiction, suppose that for any prime $p\leq 4\log N,$ $p$ is divisible by more than $d/2$ of the $a_i$'s, then
    \[\prod_{i=1}^d a_i \geq \prod_{p\leq 4\log N} p^{d/2}. \]
    By the prime number theorem,
    \[\sum_{p\leq 4\log N} \log p = 4\log N + o(\log N) > 3\log N\]
    when $N$ is sufficiently large. Hence
    \[\prod_{i=1}^d a_i \geq \exp\bigg(\frac{3d\log N}{2}\bigg) = N^{3d/2}.\]
    On the other hand, $\prod_i a_i\leq N^d$, which is a contradiction. This finishes the proof of the claim.

    Now pick such a prime $p \leq 4\log N$ and without loss of generality, assume $p\nmid a_i$ for all $1\leq i\leq \lceil d/2 \rceil.$ Consider $a_0, a_1, \ldots, a_{\lceil d/2 \rceil}$ as residue classes modulo $p^2.$ It's easy to see they satisfy the condition of Proposition~\ref{prop:local}. Therefore 
    \[\lceil d/2 \rceil \leq Kp,\]
    which implies $d \leq 8K\log N.$
\end{proof}
\fi
\begin{proof}
    (1). Given $a_1, a_2, \ldots, a_d \in [N],$ we claim that there is a prime $p\in\mathcal{P}$ so that $p\leq C\log N$ and $p$ does not divide at least $(C'-1)d/(C'+1)$ of the $a_i$'s. For the sake of contradiction, suppose that for any prime $p\in \mathcal{P}$, $p\leq C\log N$, $p$ divides more than $2d/(C'+1)$ of the $a_i$'s. Then
    \[\prod_{i=1}^d a_i \geq \prod_{\substack{p\leq C\log N\\p\in \mathcal{P}}} p^{2d/(C'+1)}. \]
    By our assumption,
    \[\sum_{\substack{p\leq C\log N\\p\in \mathcal{P}}} \log p \geq C'\log N. \]
    Hence
    \[
    N^d\geq \prod_{i=1}^d a_i \geq \exp\bigg(\frac{2C'd\log N}{C'+1}\bigg) = N^{\frac{2C'd}{C'+1}},\]
    which is a contradiction since $C'>1$. This finishes the proof of the claim.

    Now pick such a prime $p_0 \in \mathcal{P}$ and without loss of generality, assume $p_0\nmid a_i$ for all $1\leq i\leq \lceil (C'-1)d/(C'+1) \rceil.$ Consider $a_0, a_1, \ldots, a_{\lceil (C'-1)d/(C'+1) \rceil}$ as residue classes modulo $p_0^2.$ It is easy to see that they satisfy the condition of Proposition~\ref{prop:local}. Therefore Proposition~\ref{prop:local}(3) implies that
    \[\lceil (C'-1)d/(C'+1) \rceil \leq K'p_0,\]
    where $K'$ is an absolute constant, 
    that is, \[d \leq \frac{C'+1}{C'-1}K'p_0\ll \frac{C(C'+1)}{C'-1}\log N.\]

    (2). Let $\mathcal{B}$ be the set of primes $p\in\mathcal{P}$ such that at least half of the $a_i$'s are divisible by $p$. We claim that $\mathcal{B}$ is ``small". Indeed, we have
$$
N^{|A|}\geq \prod_{a \in A} a \geq \prod_{p \in \mathcal{B}} p^{|A|/2}
$$
and thus $\sum_{p \in \mathcal{B}} \log p\leq 2\log N$. Note that the function $\frac{\log x}{x}$ is decreasing when $x\geq 3$, thus 
$$
\sum_{p \in \mathcal{B}} \frac{\log p}{p}\leq \sum_{\substack{p \leq X\\p\in \mathcal{P}}} \frac{\log p}{p},
$$
where $X$ is the largest prime such that $\sum_{p\leq X,p\in \mathcal{P}}\log p\leq 2\log N$. From inequality~\eqref{eq:powerful_cond2}, we have $X\leq C\log N$ and thus 
\begin{equation}\label{eq:powerful_impl}
\sum_{p \in \mathcal{B}} \frac{\log p}{p}\leq \sum_{\substack{p\leq C\log N\\p\in\mathcal{P}}}\frac{\log p}{p}.
\end{equation}

Now we take advantage of the assumption that the elements in $A$ are distinct. Assume that $|A|>(\log N)^{0.9}$, for otherwise we are done. For each prime $p\in \mathcal{P}$ and each $0\leq i \leq p-1$, let $r(i,p)$ be the number of elements in $A$ that are congruent to $i \pmod p$. Following the proof of Gallagher's larger sieve~\cite{G71}, we have
$$
N^{|A|^2} \geq \prod_{a,b \in A, a \neq b} |a-b|\geq \prod_{p \in \mathcal{P}\setminus\mathcal{B}} p^{\sum_{i=1}^{p-1}r(i,p)(r(i,p)-1)}
$$

Let $p \in \mathcal{P}\setminus\mathcal{B}$ with $p\leq (\log N)^{1.5}$. Let $A'=\{x\in A: p \nmid x\}$. Then we have $|A'|\geq |A|/2\geq p^{0.51}$. For each $k\in \N$, let $\alpha(k)$ be the number of elements $0\leq a\leq p-1$ such that
at least $k$ elements in $A'$ are congruent to $a$ modulo $p$. By applying Proposition~\ref{prop:local}(4) to $A'$, we get
\begin{align*}
\sum_{i=1}^{p-1}r(i,p)(r(i,p)-1)&=\sum_{k\geq 1} k(k-1)(\alpha(k)-\alpha(k+1))\\
&=\sum_{k\geq 2} 2(k-1)\alpha(k) \gg \sum_{k\geq 2} k\alpha(k)\gg |A'|^3/p \gg |A|^3/p,
\end{align*}
where the implied constant is absolute.

We thus have
$$
\exp(|A^2|\log N)=N^{|A|^2} \geq \prod_{\substack{p \leq (\log N)^{1.5}\\ p \in \mathcal{P}\setminus\mathcal{B}}} p^{c|A|^3/p}=\exp \left(c|A|^3 \sum_{\substack{p \leq (\log N)^{1.5}\\ p \in \mathcal{P}\setminus\mathcal{B}}} \frac{\log p}{p}\right)
$$
for some constant $c>0$. It then follows from inequalities \eqref{eq:powerful_cond2} and \eqref{eq:powerful_impl} that
\begin{align*}\sum_{\substack{p \leq (\log N)^{1.5}\\ p \in \mathcal{P}\setminus\mathcal{B}}} \frac{\log p}{p} &\geq \sum_{\substack{p \leq (\log N)^{1.5}\\p\in\mathcal{P}}} \frac{\log p}{p}-\sum_{p \in \mathcal{B}} \frac{\log p}{p}\\
&\geq \sum_{\substack{C\log N< p\leq (\log N)^{1.5}\\p\in \mathcal{P}}}\frac{\log p}{p} \geq C'\log \log N.
\end{align*}
Combining the above two estimates, we conclude that
\[
|A|\ll \frac{\log N}{C'\log \log N},
\]
where the implied constant is absolute.
\end{proof}

\section{The second framework: local to global}\label{sec:second}
In this section, we prove the theorems in our second framework by adapting the techniques used in Section~\ref{sec:powerful}. 
\subsection{Local version}
We begin by proving the following lemma related to zero-sum-free sequences in an arbitrary finite abelian group. 
\begin{lem}\label{lem:Aq}
    Let $G$ be a finite abelian group, $A = \{a_1,a_2,\ldots,a_d\}\subseteq G$ be a multiset where each $a_k$ has multiplicity $n_k.$ For any positive integer $m\in\N,$ define
\begin{equation}\label{eq:fm}
\alpha(m) = |\{1\leq k \leq d : n_k \geq m\}|.
\end{equation}
If $0\notin \Sigma^*(A)$, then $\alpha(m) \le 32\sqrt{|G|/m}$ for all $m$.
\end{lem}
\begin{proof}
    We prove by induction on $|G|.$ For $|G|=1,2,$ the result is trivial. Suppose the statement of the lemma is true for all abelian groups with order less than $q.$ Fix an abelian group $G$ with $|G|=q\geq 3$ and $m\in \N.$ Without loss of generality, we may assume $m<q$ since otherwise $\alpha(m)=0$ and the inequality follows immediately. Choose a submultiset $A'\subseteq A$ such that $A' = \{a_1',\ldots,a_{\alpha(m)}'\}$, where each $a_i' \in G \setminus \{0\}$ and $a_i'$ has multiplicity $m.$ Let $H$ be the stabilizer of $\Sigma(A')\subseteq G.$ If $H=G,$ then $\Sigma(A')=G,$ and from our assumption we must have $\Sigma^*(A')=G\setminus\{0\}.$ It follows that $\Sigma^*(A'\setminus\{a_1'\})\neq G\setminus\{0\},$ which implies $\Sigma(A'\setminus\{a_1'\})\neq G$ and $\operatorname{stab}(\Sigma(A'\setminus\{a_1'\}))\neq G.$ Therefore we may assume $H\neq G$ at the cost of reducing $\alpha(m)$ by $1.$  
    
    Now consider the multiset $A'\cap H = \{b_1,\ldots,b_s\},$ where each element $b_i$ has multiplicity $m.$ Since $A'\cap H$ is zero-sum-free in $G,$ it is also zero-sum-free in $H.$ Moreover, $[G:H]\geq 2,$ thus from induction hypothesis we get $s\leq 32\sqrt{|H|/m}\leq 32\sqrt{|G|/2m}.$ 
    
    Let $r=[G:H]$, $c_1, \ldots, c_r$ be the coset representatives in the quotient group $G/H$. For each $1\leq i \leq r$, let $m_i$ be the number of elements $x$ in $A'\setminus H$ such that $x$ and $c_i$ belongs to the same $H$-coset, that is, $x+H=c_i+H$. For any $k\in\N,$ we define $\beta(k)=|\{1\leq i \leq r : m_i\geq k\}|.$ Then 
    \[\sum_{k\geq 1}\beta(k) = \sum_{i=1}^r m_i = m\alpha(m)-|A'\cap H| \geq m\alpha(m)-32\sqrt{|G|m/2}.\]
    If $m\alpha(m)\leq 32\sqrt{|G|m/2},$ then $\alpha(m)\leq 32\sqrt{|G|/2m}<32\sqrt{|G|/m}$ and we are done. Otherwise, for each $1\leq i\leq r,$ there are at most $|H|$ distinct elements in $G$ that belong to the coset $c_i+H$. Hence $m_i \leq m|H|,$ which implies $\beta(k)=0$ for all $k>m|H|.$ Now, by Cauchy-Schwarz, we get
    \[m|H| \cdot \sum_{k\geq 1}\beta(k)^2 \geq (\sum_{k\geq 1} \beta(k))^2 \geq m^2\bigg(\alpha(m)-32\sqrt{|G|/2m}\bigg)^2.\]
    By Lemma~\ref{lem:DGMS}, we have
    \[|G| \geq |\Sigma(A')| \geq \frac{|H|}{64}\sum_{k\geq 1}\beta(k)^2.\]
    Therefore
    \[m^2\bigg(\alpha(m)-32\sqrt{|G|/2m}\bigg)^2\leq m|H|\cdot\frac{64|G|}{|H|}\Rightarrow \alpha(m)\leq (8+16\sqrt{2})\sqrt{|G|/m}.\]
    In all cases, we have
    \[\alpha(m)\leq 1+(8+16\sqrt{2})\sqrt{|G|/m}\leq (9+16\sqrt{2})\sqrt{|G|/m}<32\sqrt{|G|/m}.\]
    This finishes the induction step.
\end{proof}

Lemma~\ref{lem:Aq} implies the following corollary, first proved by Szemer\'edi \cite{S70}, readily.

\begin{cor}\label{cor:Aq}
Let $G$ be a finite abelian group. If $A\subseteq G$ is a multiset such that $A$ contains at least $32\sqrt{|G|}$ distinct elements from $G$, then $\Sigma^*(A)$ contains $0\in G$.     
\end{cor}

Corollary~\ref{cor:Aq} is sharp up to the constant since one can take $G=\Z_q$ and $A=[\lfloor\sqrt{q}\rfloor]$. This also shows that Lemma~\ref{lem:Aq} is sharp up to the constant without extra assumptions.

Next, we use Lemma~\ref{lem:Aq} to deduce the following theorem, which refines a result of S{\'a}rk{\"o}zy~\cite[Theorem 7]{S94} (we remove the $\log q$ factor from the assumptions in \cite[Theorem 7]{S94}). This theorem can also be viewed as a multiset strengthening of Szemer\'edi's theorem stated in Corollary~\ref{cor:Aq}.
\begin{thm}\label{thm:Aq}
    Let $G$ be an abelian group, $H$ be a subgroup of $G$ with finite index, and $q=[G:H].$ Suppose $A\subseteq G$ is finite. For each element $k\in G/H$, let $n_k=|\pi^{-1}(k)\cap A|,$ where $\pi:G\rightarrow G/H$ is the group homomorphism. If $|A| > 256\sqrt{q}$
    and
    \[\sum_{k\in G/H} n_k^2 < \frac{|A|^3}{10^5q},\]
    then $\pi(\Sigma^*(A))$ contains the identity.
\end{thm}
\begin{proof}
Suppose $0\notin\pi(\Sigma^*(A)).$ By viewing $A$ as a multiset in $G/H,$ we may define $\alpha(m)$ as in equation~\eqref{eq:fm}. Then Lemma~\ref{lem:Aq} applied to $G/H$ implies that $\alpha(m) \le 32\sqrt{q/m}$ for all $m \in \N$.
    Notice that 
    \[|A| = \sum_{k\in G/H} n_k = \sum_{m\ge 1} \alpha(m);\ \sum_{k\in G/H} n_k^2 = \sum_{m\geq 1}(2m-1)\alpha(m).\] 
    If $|A| > 256\sqrt{q},$ then 
    \[\sum_{k\in G/H} n_k^2 \geq \sum_{m\geq 1}m\alpha(m) \geq 32\sum_{1\leq m\leq T}\sqrt{qm} > 21\sqrt{q}\cdot T^{3/2},\]
    where $T$ is the largest integer such that $32\sum_{m\leq T}\sqrt{q/m} \leq |A|.$ Obviously $T\geq 8$. Since 
    \[32\sum_{m\leq T}\sqrt{q/m} \leq 64\sqrt{qT},\]
    we have $T\geq \lfloor |A|^2/64^2q\rfloor$ and hence
    \[\sum_{k\in G/H} n_k^2 \geq 21\sqrt{q}\cdot\frac{|A|^3}{128^3q^{3/2}} \geq \frac{|A|^3}{10^5q},\]
    a contradiction.
\end{proof}

We have similar results for incomplete sequences in $\Z_q.$
\begin{lem}\label{lem:incomplete}
    Let $q$ be a positive integer, $A=\{a_1,a_2,\ldots,a_d\}\subseteq\Z_q$ be a multiset where each $a_k$ has multiplicity $n_k$ and $\gcd(a_k,q)=1.$ For any positive integer $m\in\N,$ define 
    \[\alpha(m) = |\{1\leq k \leq d : n_k \geq m\}|.\]
    If $\Sigma^*(A)\neq\Z_q,$ then $\alpha(m)\leq 9\sqrt{q/m}$ for all $m.$
\end{lem}
\begin{proof}
    From our assumption $\gcd(a_k,q)=1,$ we may assume $m<q$ since otherwise $\alpha(m)=0$ and the inequality follows immediately. Choose a submultiset $A'\subseteq A$ such that $A' = \{a_1',\ldots,a_{\alpha(m)}'\}$, where each $a_i' \in \Z_q \setminus \{0\}$ and $a_i'$ has multiplicity $m.$ Let $H$ be the stabilizer of $\Sigma(A')\subseteq \Z_q.$ Same as the proof of Lemma~\ref{lem:Aq}, we may assume $H\neq \Z_q$ at the cost of reducing $\alpha(m)$ by $1.$ Then from the assumption on $\gcd$ again we conclude $A'\cap H=\emptyset.$ 
    
    Let $r=[\Z_q:H]$, $c_1, \ldots, c_r$ be the coset representatives in the quotient group $\Z_q/H$. For each $1\leq i \leq r$, let $m_i$ be the number of elements $x$ in $A'$ such that $x$ and $c_i$ belong to the same $H$-coset. For any $k\in\N,$ we define $\beta(k)=|\{1\leq i \leq r : m_i\geq k\}|.$ Then 
    \[\sum_{k\geq 1}\beta(k) = \sum_{i=1}^r m_i = m\alpha(m).\]
    For each $1\leq i\leq r,$ there are at most $|H|$ distinct elements in $\Z_q$ that belong to the coset $c_i+H$. Hence $m_i \leq m|H|,$ which implies $\beta(k)=0$ for all $k>m|H|.$ Now by Cauchy-Schwarz, we get
    \[m|H| \cdot \sum_{k\geq 1}\beta(k)^2 \geq \bigg(\sum_{k\geq 1} \beta(k)\bigg)^2 = m^2\alpha(m)^2.\]
    By Lemma~\ref{lem:DGMS}, we have
    \[q \geq |\Sigma(A')| \geq \frac{|H|}{64}\sum_{k\geq 1}\beta(k)^2.\]
    Therefore
    \[m^2\alpha(m)^2\leq m|H|\cdot\frac{64q}{|H|}\Rightarrow \alpha(m)\leq 8\sqrt{q/m}.\]
    In all cases, we have
    \[\alpha(m)\leq 1+8\sqrt{q/m}\leq 9\sqrt{q/m}.\qedhere\]
\end{proof}
\begin{thm}\label{thm:incom}
    Let $q\in\N$, $A$ be a finite set of positive integers that are coprime to $q$. For each positive integer $k,$ let $n_k$ denote the number of elements of $A$ that are congruent to $k$ modulo $q.$ If $|A| > 72\sqrt{q}$
    and
    \[\sum_{k=1}^q n_k^2 < \frac{|A|^3}{7776q},\]
    then $\Sigma^*(A)$ covers all residue classes modulo $q.$
\end{thm}
\begin{proof}
    Let $\alpha(m)$ be defined as in equation~\eqref{eq:fm}. Then Lemma~\ref{lem:incomplete} implies that $\alpha(m) \le 9\sqrt{q/m}$ for all $m \in \N$.
    Notice that 
    \[|A| = \sum_{k=1}^q n_k = \sum_{m\ge 1} \alpha(m);\ \sum_{k=1}^q n_k^2 = \sum_{m\geq 1}(2m-1)\alpha(m).\] 
    If $|A| > 72\sqrt{q},$ then 
    \[\sum_{k=1}^q n_k^2 \geq \sum_{m\geq 1}m\alpha(m) \geq 9\sum_{1\leq m\leq T}\sqrt{qm} > 6\sqrt{q}\cdot T^{3/2},\]
    where $T$ is the largest integer such that $9\sum_{m\leq T}\sqrt{q/m} \leq |A|.$ Obviously $T\geq 8$. Since 
    \[9\sum_{m\leq T}\sqrt{q/m} \leq 18\sqrt{qT},\]
    we have $T\geq \lfloor |A|^2/18^2q\rfloor$ and hence
    \[\sum_{k=1}^q n_k^2 \geq 6\sqrt{q}\cdot\frac{|A|^3}{36^3q^{3/2}} \geq \frac{|A|^3}{7776q},\]
    a contradiction.
\end{proof}
\subsection{Proof of second framework}
Next we use Theorems~\ref{thm:Aq} and~\ref{thm:incom} and sieve techniques to prove Theorems~\ref{thm:primefactor} and~\ref{thm:main4}.
\begin{proof}[Proof of Theorem~\ref{thm:primefactor}]
Let $c=1/(256^2+1)$. For each $b\in B \subseteq [1,c|A|^2]$ and each $1\leq i \leq b$, let $n_{b,i}$ denote the number of elements of $A$ congruent to $i$ modulo $b$. 

Let $b\in B$. By assumption, $\Sigma^*(A)$ does not contain an element that is divisible by $b$. Since $c>0$ is sufficiently small, we have $|A|>256 \sqrt{c}|A|\geq 256\sqrt{b}$. Thus, by Theorem~\ref{thm:Aq}, we have
\begin{equation}\label{eq:nb^2}
\sum_{i=1}^b n_{b,i}^2 \geq \frac{|A|^3}{10^5b}.
\end{equation}
By Lemma~\ref{lem:Aq}, the number of $1\leq i \leq b$ with $n_{b,i}=1$ is at most $32\sqrt{b}$. For each $i$ with $n_{b,i}\geq 2$, note that $\binom{n_{b,i}}{2}\geq \frac{1}{4}n_{b,i}^2$. Thus, it follows from inequality~\eqref{eq:nb^2} that
$$
\sum_{i=1}^b \binom{n_{b,i}}{2}\geq \frac{1}{4} \bigg(\sum_{i=1}^b n_{b,i}^2 -32\sqrt{b}\bigg)\geq \frac{|A|^3}{4 \cdot 10^5 b}-8\sqrt{b}\geq \frac{|A|^3}{Cb}
$$
for some absolute constant $C$ (for example, we can take $C=5 \cdot 10^5$). 

Since the elements in $B$ are pairwise coprime, following the proof of a generalization of Gallagher's larger sieve \cite[Lemma 3]{ESS94}, we have
$$
N^{|A|^2} \geq \prod_{a,a'\in A, a<a'} (a'-a) \geq \prod_{b\in B} b^{\sum_{i=1}^b \binom{n_{b,i}}{2}}\geq \prod_{b\in B} b^{|A|^3/Cb}.
$$
Taking the logarithm on both sides of the above inequality, it follows that
$$
|A|^2\log N\geq \frac{|A|^3}{C} \sum_{b\in B} \frac{\log b}{b},
$$
as required.
\end{proof}

\begin{proof}[Proof of Theorem~\ref{thm:main4}]
    Let $c=(\frac{\eps}{73(1+\eps)})^2.$ For any $b\in B,$ define $A(b) = \{a\in A: \gcd(a,b)=1\}.$ Let $B' = \{b\in B: |A(b)|\geq \eps|A|/(1+\eps)\}.$ For each $b\in B'\subseteq[1,c|A|^2]$ and each $1\leq i\leq b,$ let $n_{b,i}$ denote the number of elements of $A$ congruent to $i$ modulo $b.$ 

    Let $b\in B'$. By assumption, $b$ is a prime power and $\Sigma^*(A(b))$ does not cover all residue classes modulo $b$. Since $c>0$ is sufficiently small, we have $$|A(b)|\geq \eps|A|/(1+\eps)>72 \sqrt{c}|A|\geq 72\sqrt{b}.$$ Thus, by Theorem~\ref{thm:incom}, we have
\begin{equation}\label{eq:nb^22}
\sum_{i=1}^b n_{b,i}^2 \geq \frac{|A(b)|^3}{7776b}\geq \bigg(\frac{\eps}{1+\eps}\bigg)^3\frac{|A|^3}{7776b}.
\end{equation}
By Lemma~\ref{lem:incomplete}, the number of $1\leq i \leq b$ with $n_{b,i}=1$ is at most $9\sqrt{b}$. For each $i$ with $n_{b,i}\geq 2$, note that $\binom{n_{b,i}}{2}\geq \frac{1}{4}n_{b,i}^2$. Thus, it follows from inequality~\eqref{eq:nb^22} that
$$
\sum_{i=1}^b \binom{n_{b,i}}{2}\geq \frac{1}{4} \bigg(\sum_{i=1}^b n_{b,i}^2 -9\sqrt{b}\bigg)\geq \bigg(\frac{\eps}{1+\eps}\bigg)^3\frac{|A|^3}{31104 b}-\frac{9}{4}\sqrt{b}\geq \frac{|A|^3}{Cb}
$$
for some constant $C=C(\eps)$ (for example, we can take $C=4\cdot10^4(\frac{1+\eps}{\eps})^3$). 

Since the elements in $B'$ are pairwise coprime, following the proof of a generalization of Gallagher's larger sieve \cite[Lemma 3]{ESS94}, we have
$$
N^{|A|^2} \geq \prod_{a,a'\in A, a<a'} (a'-a) \geq \prod_{b\in B'} b^{\sum_{i=1}^b \binom{n_{b,i}}{2}}\geq \prod_{b\in B'} b^{|A|^3/Cb}.
$$
Taking the logarithm on both sides of the above inequality, it follows that
$$
|A|^2\log N\geq \frac{|A|^3}{C} \sum_{b\in B'} \frac{\log b}{b},
$$
which implies 
$$
|A|\sum_{b\in B'}\frac{\log b}{b}\leq C\log N.
$$
To get the desired inequality, we need to obtain a lower bound on $\sum_{b\in B'}\log b/b.$ Notice that for any $b\in B\setminus B',$ $|\{a\in A: \gcd(a,b)>1\}| = |A\setminus A(b)| > |A|/(1+\eps).$ Hence 
$$
|A|\log N \geq \sum_{a\in A}\log a \geq \frac{|A|}{1+\eps}\sum_{b\in B\setminus B'}\Lambda(b).
$$
Equivalently, it suffices to find the maximum value of $\sum_{b\in \tilde B}(\log b)/b$ subject to the constraint $\sum_{b\in \tilde B}\Lambda(b)\leq (1+\eps)\log N.$ Since the function $(\log x)/x$ is decreasing when $x\geq 3,$ it is clear that the maximum is achieved when $\tilde{B} = B\cap[1,X],$ where $X$ is the largest integer such that 
$$
\sum_{b\in B,b\leq X}\Lambda(b)\leq (1+\eps)\log N.
$$
Therefore
$$
|A|\sum_{b\in B''}\frac{\log b}{b} \leq |A|\sum_{b\in B'}\frac{\log b}{b}\leq C\log N,
$$
where $B''=B\setminus[1,X+1].$
\end{proof}

To prove Theorem~\ref{thm:multi_incom}, we need the following lemma by Vu~\cite{V07}.

\begin{lem}[Vu]\label{lem:Vu}
    Let $n$ be a positive integer. Let $A$ be a multiset consisting of some elements in $\Z_n$ that are coprime to $n$. If $|A|\geq n$, then $\Sigma^*(A)=\Z_n$.
\end{lem}

\begin{proof}[Proof of Theorem~\ref{thm:multi_incom}]
For any $b\in B,$ define $A(b) = \{a\in A: \gcd(a,b)=1\}.$ Let $B' = \{b\in B: |A(b)|\geq \eps|A|/(1+\eps)\}.$ We claim that $B'\neq\emptyset.$ Note that for each $b\in B\setminus B',$ we have $$|\{a\in A: \gcd(a,b)>1\}| = |A\setminus A(b)| > |A|/(1+\eps).$$ Hence 
$$
|A|\log N \geq \sum_{a\in A}\log a \geq \frac{|A|}{1+\eps}\sum_{b\in B\setminus B'}\Lambda(b),
$$
which implies \[\sum_{b\in B\setminus B'}\Lambda(b)\leq (1+\eps)\log N.\]
Given the assumption that $\sum_{b\in B}\Lambda(b)>(1+\eps)\log N,$ there must exist some $b_0\in B'.$ By viewing $A(b_0)$ as a multiset modulo $b_0,$ it follows from Lemma~\ref{lem:Vu} that $|A(b_0)|<b_0.$ Hence
\[|A|\leq \frac{1+\eps}{\eps}|A(b_0)|\leq (1+\eps^{-1})b_0\leq (1+\eps^{-1})\max_{b\in B} b.\qedhere\]
\end{proof}

\section{Hilbert cubes in perfect powers}\label{sec:perfectpower}

In this section, we apply our first framework to prove the results stated in Section~\ref{sec:intro_perfectpower}. To apply the framework, we first estimate the length of an arithmetic progression contained in $\PP\cap [N]$, and then study sumsets in $\PP \cap [N]$. Recall $\PP$ stands for the set of perfect powers.

\subsection{Preliminaries}
There are many conjectures related to sums of perfect powers. Next, we recall the ABC conjecture and the Lander--Parkin--Selfridge conjecture~\cite{LPS67}, which will be assumed in some results in this section. Recall that for a nonzero integer $n$, its \emph{radical} is  $\rad(n)=\prod_{p \mid n}p$. 

\begin{conj}[ABC conjecture]\label{conj:ABC}
For each $\eps>0$, there is a constant $K_\eps$, such that whenever $a,b,c$ are nonzero integers with $\gcd(a,b,c)=1$ and $a+b=c$, we have 
$$
\max \{|a|,|b|,|c|\} \leq K_{\eps} \rad (abc)^{1+\eps}.
$$    
\end{conj}

\begin{conj}[Lander--Parkin--Selfridge conjecture]\label{conj:LPS}
Let $m,n,k$ be positive integers. If $$\sum _{i=1}^{n}a_{i}^{k}=\sum _{j=1}^{m}b_{j}^{k},$$ where $a_1,a_2, \ldots, a_n, b_1, b_2, \ldots, b_m$ are positive integers such that $a_i \neq b_j$ for all $1\leq i \leq n$ and $1\leq j \leq m$, then $m+n \geq k$.
\end{conj}

For our purpose, we only need the following special case of Conjecture~\ref{conj:LPS}. The case $k=5$ of this conjecture was first formulated by Erd\H os. 
\begin{conj}[Lander--Parkin--Selfridge conjecture, special case]\label{conj:LPS2}
There is $k_0$, such that whenever $k\geq k_0$ is an integer, there do not exist positive integers $a_1,a_2,b_1,b_2$ such that $a_1^k+a_2^k=b_1^k+b_2^k$ and $\{a_1,a_2\}\neq \{b_1,b_2\}$. Equivalently, if $k$ is sufficiently large, then the set of $k$-th powers forms a Sidon set.   
\end{conj}

In Section~\ref{sec:sumsetshiftpower}, we will use some basic tools from graph theory. Recall that a \emph{bipartite graph $G$ with bipartition $(A,B)$} is a graph with vertex set $A\cup B$ such that no two vertices in $A$ (resp. $B$) are adjacent. $K_{s,t}$ denotes a complete bipartite graph with bipartition $(A,B)$, where $|A|=s$ and $|B|=t$, that is, there is an edge between $a$ and $b$ for all $a\in A$ and $b \in B$. The following is the K\"ov\'ari--S\'os--Tur\'an theorem \cite{KST54}, a fundamental result in extremal graph theory.  

\begin{lem}[K\"ov\'ari--S\'os--Tur\'an theorem]\label{lem:KST}\
Let $G$ be a bipartite graph with bipartition $(U,V)$  such that $|U|=m$ and $|V|=n$. Assume that there does not exist a set $X\subseteq U$ with size $s$ and a set $Y\subseteq V$ with size $t$, such that $x$ and $y$ are adjacent for all $x\in X$ and $y\in Y$. Then the number of edges of $G$ is at most 
$(s-1)^{1/t}(n-t+1)m^{1-1/t}+(t-1)m.$
\end{lem}

\subsection{Number of $k$-th powers in an arithmetic progression: Proof of Theorem~\ref{thm:ABCQ_k}}
Recall that for each $k\geq 2$, $N,q\geq 1$ and $a\in \Z$, $Q_k(N;q,a)$ is the number of $k$-th powers in the arithmetic progression $a+q,a+2q, \ldots, a+N q$. Also recall that $Q_k(N)=\max_{q\geq 1, a \geq 0} Q_k(N;q,a)$.

In this subsection, we prove Theorem~\ref{thm:ABCQ_k}. A key ingredient of our proof is the following result of Bourgain and Demeter \cite{BD18} on the upper bound of $Q_k(N;q,a)$. For the sake of completeness, we make their upper bound explicit.

\begin{prop}[Bourgain-Demeter]\label{prop:BD}
For each polynomial $P_k \in \Z[x]$ of degree $k\geq 1$ and each $a\in \Z$ and $N,q\geq 1$, we have
$$
\#\{ 1\leq j \leq N: a+jq \in P_k(\Z)\} \leq (4\tau(q))^{k-1} N^{1/k},
$$
where $\tau(q)$ denotes the number of divisors of $q$.  In particular, \[Q_k(N;q,a)\leq (4 \tau(q))^{k-1} N^{1/k}\] for each $k\geq 2$, $a\geq 0$, and $q,N\geq 1$.
\end{prop}
\begin{proof}
We prove the result by induction on $k$. The case $k=1$ is trivial. Assume the statement holds for $k$ and let $P_{k+1}$ be a polynomial of degree $k+1$. If there is no $t\in \Z$ such that $P_{k+1}(t) \in \{a+q, a+2q, \ldots, a+Nq\}$, then we are done. Next assume that $t_0$ is the smallest integer such that $P_{k+1}(t_0)\in \{a+q, a+2q, \ldots, a+Nq\}$. Then we can write $$P_{k+1}(x)-P_{k+1}(t_0)=(x-t_0)P_k(x)$$ for some polynomial $P_k \in Z[x]$ with degree $k$. It follows that for each $t\in \Z$ with $P_{k+1}(t)\neq P_{k+1}(t_0)$ and $P_{k+1}(t) \in \{a+q, a+2q, \ldots, a+Nq\}$, we have $$P_{k+1}(t)-P_{k+1}(t_0)=(t-t_0)P_k(t)\in q\Z,$$
where $t-t_0=n_1 q_1$ and $P_k(t)=n_2q_2$ for some nonzero integers $n_1, n_2, q_1, q_2$ with $q_1q_2=q$, $n_1\geq 1$, and $n_1|n_2|\leq N$. It follows that either $n_1\leq N^{1/(k+1)}$ or $|n_2|\leq N^{k/(k+1)}$.

Fix a pair $(q_1, q_2)$. In the first case, that is, $n_1\leq N^{1/(k+1)}$, there are at most $N^{1/(k+1)}$ possible values of $t$. In the second case, by considering the case $n_2>0$ and $n_2<0$ separately, there are at most $$2 (4\tau(q_2))^{k-1} (N^{k/(k+1)})^{1/k}\leq 2(4 \tau(q))^{k-1} N^{1/(k+1)}$$ possible values of $t$ by the inductive hypothesis. 

Since these are $\tau(q)$ possible pairs $(q_1,q_2)$, we conclude that 
$$
\#\{ 1\leq j \leq N: a+jq \in P_k(\Z)\} \leq \tau(q) (N^{1/(k+1)}+2 (4\tau(q))^{k-1} N^{1/(k+1)})+1 \leq (4\tau(q))^k N^{1/(k+1)},
$$
as required.
\end{proof}

The following lemma shows to estimate $Q_k(N)$, it suffices to consider those $Q_k(N;q,a)$ with $\gcd(a,q)=1$.

\begin{lem}\label{lem:gcd1}
Let $k\geq 2$ and $N\geq 1$. If $a\in \Z$ and $q\geq 1$ such that $\gcd(a,q)>1$, then there exist $a'\in \Z$ and $q'\geq 1$ such that $Q_k(N;q,a)\leq Q_k(N;q',a')$ and $\gcd(a',q')<\gcd(a,q)$.
\end{lem}
\begin{proof}
We may assume that $Q_k(N;q,a)\geq 2$, otherwise the lemma is trivial. Since $\gcd(a,q)>1$, we can find a prime $p \mid \gcd(a,q)$. Note that if $p^k \mid \gcd(a,q)$, then $Q_k(N;q,a)=Q_k(N;q/p^k, a/p^k)$ and we are done. Next assume that $p^k \nmid \gcd(a,q)$.  Let $j$ be the smallest integer such that $a+jq$ is a $k$-th power. Let $j'>j$ be an integer such that $a+j'q$ is a $k$-th power. Since $p \mid \gcd(a+jq, a+j'q)$, it follows that $p^k \mid \gcd(a+jq, a+j'q) \mid (j'-j)q$. Thus, $\tilde{q} \mid (j'-j)q$, where $\tilde{q}=\lcm(p^k,q)$. This shows that all $k$-th powers in the progression $a+q, a+2q, \ldots, a+Nq$ are actually contained in the progression $a+jq, a+jq+\tilde{q}, a+jq+2\tilde{q}, \ldots, a+jq+(N-1)\tilde{q}$, that is, $Q_k(N;q,a)\leq Q_k(N;\tilde{q}, a+jq-\tilde{q})$. Since $p^k \mid \gcd(\tilde{q}, a+jq)$, by the same observation, we have $Q_k(N;q,a)\leq Q_k(N;q',a')$, where $q'=\tilde{q}/p^k$ and $a'=(a+jq-\tilde{q})/p^k$. Note that $$\gcd(a',q')=\gcd(a+jq-\tilde{q}, \tilde{q})/p^k=\gcd(a+jq, \lcm(p^k,q))/p^k \mid \gcd(a+jq, q)=\gcd(a,q).$$ Also, note that $p^k \nmid q$ since $p^k \mid (a+jq)$ and $p^k \nmid \gcd(a,q)$. It follows that $p\nmid q'$ and thus $\gcd(a',q')\leq \gcd(a,q)/p$, as required.
\end{proof}

\begin{lem}\label{lem:boundq}
Assume the ABC conjecture. If $k\geq 4$, $q\geq 1, a\geq 0$, $\gcd(a,q)=1$, and 
$Q_k(N;q,a)\geq 3$, then $q\leq KN^{60}$, where $K$ is an absolute constant.  
\end{lem}
\begin{proof}
Assume that $a+j_1 q, a+j_2 q, a+j_3 q$ are $k$-th powers with $1\leq j_1<j_2<j_3\leq N$. Say $a+j_i q= x_i^k$ for $1\leq i \leq 3$. Then 
$$
(j_3-j_2)x_1^k+(j_1-j_3)x_2^k=(j_1-j_2)a+j_3(j_1-j_2)q=(j_1-j_2)x_3^k.
$$
On the other hand, since $\gcd(a,q)=1$, it follows that $$\gcd(x_1^k, x_2^k)=\gcd(a+j_1 q, a+j_2 q)=\gcd(a+j_1q, (j_2-j_1)q)=\gcd(a+j_1q, j_2-j_1)\leq j_2-j_1\leq N$$
and thus $$D=\gcd((j_3-j_2)x_1^k, (j_1-j_3)x_2^k)\leq (j_3-j_2)(j_3-j_1)\gcd(x_1^k, x_2^k)\leq N^3.$$ Let $$A=(j_3-j_2)x_1^k/D, B=(j_1-j_3)x_2^k/D, C=(j_1-j_2)x_3^k/D.$$ Then we have $A+B=C$ and $\gcd(A,B)=1$. Let $K=K_{1/6}$ from the ABC conjecture stated in Conjecture~\ref{conj:ABC}. It follows that
$$
\frac{x_1^k x_2^k x_3^k}{N^9} \leq |ABC|\leq (K_{1/6} \rad (ABC)^{1+\frac{1}{6}})^3 \leq K_{1/6}^3 (N^3x_1x_2x_3)^{7/2}.
$$
Thus,
$$
q^{3/8} \leq q^{3(k-7/2)/k} \leq (x_1x_2x_3)^{k-7/2}\leq K_{1/6}^3 N^{20},
$$
and it follows that $q\ll N^{60}$, as required.
\end{proof}

Now we are ready to conclude the proof of Theorem~\ref{thm:ABCQ_k}.

\begin{proof}[Proof of Theorem~\ref{thm:ABCQ_k}]
By Lemma~\ref{lem:gcd1} and Lemma~\ref{lem:boundq}, if $Q_k(N)\geq 3$, then $$Q_k(N)=\max \{Q_k(N;q,a): 1\leq q\leq KN^{60}, a\geq 0, \gcd(a,q)=1\},$$
where $K$ is the absolute constant from Lemma~\ref{lem:boundq}. The theorem then follows from applying Proposition~\ref{prop:BD} for each $Q_k(N;q,a)$ in the above set, and the well-known upper bound on the divisor function \cite{NR83}.
\end{proof}

Observe that the upper bound on $Q_k(N)$ is trivial in Theorem~\ref{thm:ABCQ_k} when $k>C'\log \log N$ for some constant $C'$. For our applications in perfect powers, $k$ is not fixed, and we record the following bound that is more suitable for this purpose.

\begin{prop}\label{prop:61/k}
Assume the ABC conjecture. We have $Q_k(N)\leq CN^{61/k}$ for all $k\geq 2$ and $N\geq 1$, where $C$ is an absolute constant.
\end{prop}
\begin{proof}
The result is trivial for $k\leq 60$, so we may assume that $k\geq 61$. Let $K$ be the absolute constant from Lemma~\ref{lem:boundq}. Note that if $a\geq 0$ and $q\leq KN^{60}$, then $\{a+jq: 1\leq j \leq N\} \subseteq \{a+\ell: 1\leq \ell \leq KN^{61}\}$ and thus $$Q_k(N;q,a)\leq (a+KN^{61})^{1/k}-a^{1/k}+1\leq (KN^{61})^{1/k}+1\leq 2KN^{61/k}.$$ The proposition then follows from Lemma~\ref{lem:gcd1} and Lemma~\ref{lem:boundq}.
\end{proof}
\begin{rem}
Here, we do not attempt to optimize the constant $61$ in the exponent of the above proposition. It can certainly be improved by optimizing the parameters in Lemma~\ref{lem:boundq} and the above proof. 
\end{rem}

We end the subsection with a brief remark on another upper bound on $Q_k(N)$ conditional on Conjecture~\ref{conj:LPS2}.

\begin{rem}\label{rem:LPS_sidon}
If $k\geq 5$, then Conjecture~\ref{conj:LPS2} implies that $Q_k(N)\leq \sqrt{N}+N^{1/4}+1$. Indeed, consider an arithmetic progression $a+d, a+2d, \ldots, a+N d$ and let $S=\{1\leq j \leq N: a+jd \text{ is a $k$-th power}\}.$ Then Conjecture~\ref{conj:LPS2} implies that $S$ is a Sidon set and thus the  upper bound on $|S|$ follows from a classical result of Lindstr\"om \cite{L69} on Sidon sets contained in $[N]$.   
\end{rem}

\subsection{Arithmetic progressions in perfect powers: Proof of Theorem~\ref{thm:APperfectpower}}

As mentioned in the introduction, under the ABC conjecture, Hajdu \cite{H04} showed that if an arithmetic progression $\{a+qj: 1\leq j \leq \ell\}$ is contained in the set of perfect powers, then $\ell$ can be bounded in terms of a tower type in $\gcd(a,q)$. In the next proposition, we show that if $\gcd(a,q)>1$, then his bound can be significantly improved unconditionally; we also obtain an even better upper bound conditional on the ABC conjecture. 

\begin{prop}\label{prop:APperfectpower}
Assume that an arithmetic progression $\{a+qj: 1\leq j \leq \ell\}$ is contained in the set of perfect powers with $\Delta=\gcd(a,q)>1$. Let $r$ be a prime divisor of $\Delta$ and let $r^{\alpha} \mid\mid \Delta$. Then $\ell \leq \exp(C(\log \omega(\alpha))^9)$, where $C$ is an absolute constant and $\omega(\alpha)$ counts the number of prime divisors of $\alpha$. Moreover, we have the stronger bound $\ell \ll \omega(\alpha)$ under the ABC Conjecture. 
\end{prop}
\begin{proof}
For each $1\leq j \leq \ell$, since $a+qj$ is a perfect power, we can write $a+qj=x_j^{y_j}$, where $x_j$ is a positive integer and $y_j$ is a prime. For each $1\leq j \leq \ell-1$, since $$\gcd(x_j^{y_j}, x_{j+1}^{y_{j+1}})=\gcd(a+qj, a+q(j+1))=\gcd(a+qj, q)=\Delta,$$ it follows that 
$r^{\alpha}\mid \mid x_j^{y_j}$ or $r^{\alpha}\mid \mid x_{j+1}^{y_{j+1}}$, and thus $y_j \mid \alpha$ or $y_{j+1} \mid \alpha$. Equivalently, for each $1\leq j \leq \ell-1$, at least one of $y_j$ and $y_{j+1}$ is a prime divisor of $\alpha$. It follows that
\begin{equation}\label{eq:boundll}
\ell \ll \sum_{p \mid \alpha} Q_{p}(\ell; q, a).    
\end{equation}
Since there is no $4$-term arithmetic progression consisting of squares (see Section~\ref{sec:intro_perfectpower}), by Szemer\'edi's theorem \cite{S75}, $Q_2(\ell; q,a)=o(\ell)$. For each $p\geq 3$, since there is no $3$-term arithmetic progression consisting of $p$-th powers (see Section~\ref{sec:intro_perfectpower}), by the best-known quantitative bound on Roth's theorem due to Bloom and Sisask \cite{BS23}, there exists a positive absolute constant $c$ such that
$$
Q_{p}(\ell; q, a)\leq \frac{\ell}{\exp(c(\log \ell)^{1/9})}.
$$
Thus, inequality~\eqref{eq:boundll} implies that
$$
\ell \ll \omega(\alpha) \frac{\ell}{\exp(c(\log \ell)^{1/9})}
$$
and we conclude that \[\ell \leq \exp(C(\log \omega(\alpha))^9),\] where $C$ is an absolute constant. 

Next, we assume the ABC conjecture. We bound the right-hand side of inequality~\eqref{eq:boundll} by dividing primes into the following three groups:
\begin{equation}\label{eq:boundl}
\ell \ll \sum_{p \mid \alpha} Q_{p}(\ell; q, a)\leq \sum_{p \leq 61} Q_{p}(\ell; q, a)+ \sum_{67 \leq p \leq 61\log_2 \ell} Q_{p}(\ell; q, a)+\sum_{\substack{p \mid \alpha \\ p> 61\log_2 \ell}} Q_{p}(\ell; q, a).   
\end{equation}
For each $p\leq 61$, since there is no $4$-term arithmetic progression consisting of $p$-th powers, by Szemer\'edi's theorem, $Q_p(\ell; q,a)=o(\ell)$. It follows the first sum $\sum_{p \leq 61} Q_{p}(\ell; q, a)$ on the right-hand side of inequality~\eqref{eq:boundl} is $o(\ell)$. For the second sum and the third sum, Proposition~\ref{prop:61/k} implies that $$\sum_{67 \leq p \leq 61\log_2 \ell} Q_{p}(\ell; q, a) \ll \ell^{61/67}\log \ell, \quad \sum_{\substack{p \mid \alpha \\ p> 61\log_2 \ell}} Q_{p}(\ell; q, a)\ll \omega(\alpha).$$  
We conclude that $\ell \ll \omega(\alpha)$, as required.
\end{proof}

Now we are ready to prove Theorem~\ref{thm:APperfectpower}.

\begin{proof}[Proof of Theorem~\ref{thm:APperfectpower}]
Let $\Delta=\gcd(a_0,q)$. If $\Delta=1$, we know that $\ell \ll 1$ by Hajdu's result \cite{H04}. Next, assume that $\Delta>1$. Let $p$ be a divisor of $\Delta$ and let $p^{\alpha} \mid\mid \Delta$. Then $\alpha \ll \log a_0$. By Proposition~\ref{prop:APperfectpower}, 
$$
\ell \ll \omega(\alpha)\ll \frac{\log \alpha}{\log \log \alpha} \ll \frac{\log \log a_0}{\log \log \log a_0},
$$
as required.
\end{proof}

\subsection{Sumsets in shifted perfect powers}\label{sec:sumsetshiftpower}
In view of the general framework we have developed in Theorem~\ref{main_strong}, we need to bound $f(N)$. When $a_0=0$, by adapting the techniques used in Gyarmati, S\'ark\"ozy, and Stewart \cite[Section 4]{GSS03}, Dietmann and Elshotlz \cite[Proposition 4.6]{DE15} showed that $f(N)\ll (\log N)^L$ by taking $k=5$. However, their techniques do not extend to the general setting where $a_0>0$. We obtain a bound for $f(N)$ for a general $a_0$ in Section~\ref{subsec:CY}, and for some specific $a_0$ in Section~\ref{subsec:smallprimefactor}.

\subsubsection{A bound for arbitrary $a_0$}\label{subsec:CY}
Recently, inspired by questions related to generalizations of Diophantine tuples, the first and the third author studied product sets in shifted perfect powers \cite{CY26}. Here we study sumsets in shifted perfect powers following a similar approach. The following proposition is an additive analogue of \cite[Theorem 2.2]{CY26}.

\begin{prop}\label{prop:a0+N}
Let $B_1, B_2 \subseteq [N]$ and $a_0\geq 0$. If $a_0+B_1+B_2\subset \PP$, then $$\min\{|B_1|,|B_2|\}\ll \exp(L(\log \log (a_0+N))^2),$$
where $L$ is an absolute constant, and the implied constant is absolute. Moreover, further assuming Conjecture~\ref{conj:LPS2}, we have the stronger bound that
$$\min\{|B_1|,|B_2|\}\ll (\log (a_0+N))^L.$$
\end{prop}

The proof of the proposition is very similar to the proof of \cite[Theorem 2.2]{CY26}, so we only sketch the proof. A key ingredient of the proof is the following lemma related to sumsets in $k$-th powers, which can be viewed as an additive analogue of \cite[Theorem 2.7]{CY26}. It can be proved by combining tools from character sum estimates over finite fields and sieve methods. The proof is almost identical to the proof of \cite[Theorem 2.7]{CY26} presented in \cite[Section 6.2]{CY26} (see also \cite{EW24}), and thus we omit the details.

\begin{lem}\label{lem:loglogN}
Assume that $A_1,A_2 \subseteq [N]$ and $a_0\geq 0$ such that $a_0+A_1+A_2$ is contained in the set of $k$-th powers, where $2\leq k \leq 3\log (a_0+N)$. Then there are two absolute constants $L_1,L_2$ such that either $|A_1|\leq L_1 \log \log (a_0+N)$ or $|A_2|\leq (\log (a_0+N))^{L_2}$. 
\end{lem}

Now we are ready to prove Proposition~\ref{prop:a0+N}.
\begin{proof}[Proof of Proposition~\ref{prop:a0+N}]
Let $G$ be the complete bipartite graph with bipartition $(B_1, B_2)$. By definition, for each $b_1\in B_1$ and $b_2\in B_2$, $a_0+b_1+b_2$ is a perfect power; thus, we can write $a_0+b_1+b_2=x^p$ for some positive integer $x$ and a prime $p$, and we color the edge $b_1b_2$ by the smallest such $p$. Note that each prime $p$ we used to color some edge satisfies that $p\leq \log_2(a_0+2N)\leq 3\log M$, where $M=a_0+N$. Let $L_1, L_2$ be the two absolute constants from Lemma~\ref{lem:loglogN}. Then it follows from Lemma~\ref{lem:loglogN} that $G$ does not contain a monochromatic $K_{\lfloor (\log M)^{L_2}\rfloor+1, \lfloor L_1 \log \log M \rfloor +1,}$ as a subgraph. Thus, for each prime $p\leq 3\log M$, Lemma~\ref{lem:KST} implies that the number of edges in $G$ with color $p$ is at most 
$$
(\log M)^{L_2/(L_1 \log \log M) }
|B_2||B_1|^{1-1/\lceil L_1 \log \log M \rceil} + (L_1 \log \log M) |B_1|.
$$
Since the total number of edges in $G$ is $|B_1||B_2|$, it follows that
$$
|B_1||B_2| \leq 3\log M \bigg(((\log M)^{L_2/(L_1 \log \log M) }
|B_2||B_1|^{1-1/\lceil L_1 \log \log M \rceil} + (L_1 \log \log M) |B_1|\bigg).
$$
It follows that either $|B_2|\ll \log M \log \log M$ or 
$$
|B_1|\ll (3 \log M)^{\lceil L_1 \log \log M \rceil+L_2}\leq \exp(L_3(\log \log (a_0+N))^2),
$$
where $L_3$ is an absolute constant, as required.  

Next, let $k_0$ be the constant from Conjecture~\ref{conj:LPS2}. For each $k_0\leq p \leq 3\log M$, we claim that $G$ does not contain a monochromatic $K_{2,2}$ in color $p$ as a subgraph. Indeed, suppose otherwise there exist $a_1, a_2\in A$ and $b_1,b_2\in B$ with $a_1<a_2$ and $b_1<b_2$ such that $a_0+a_i+b_j$ is a $p$-th power for $i,j\in \{1,2\}$, then we have $$(a_0+a_1+b_1)+(a_0+a_2+b_2)=(a_0+a_1+b_2)+(a_0+a_2+b_1)$$ with $a_0+a_1+b_1<\min\{a_0+a_1+b_2,a_0+a_2+b_1\}$, violating Conjecture~\ref{conj:LPS2}. Thus, for each prime $k_0\leq p\leq 3\log M$, Lemma~\ref{lem:KST} implies that the number of edges in $G$ with color $p$ is at most 
$|B_2||B_1|^{1/2} + |B_1|.$ By considering the contributions of edges in $G$ with color $p< k_0$ and $p\geq k_0$ separately, we have
$$
|B_1||B_2|\ll k_0(\log M)^{L_2/(L_1 \log \log M) }
|B_2||B_1|^{1-1/\lceil L_1 \log \log M \rceil}+\log M (|B_2||B_1|^{1/2} + |B_1|).
$$
It then follows that $|B_2|\ll \log M$ or $|B_1|\ll (\log M)^{L_2}+(\log M)^2$, as required.
\end{proof}

We also prove the following version of Proposition~\ref{prop:a0+N}, where we obtain an improved upper bound by imposing a stronger assumption.

\begin{prop}\label{prop:g}
Let $B_1, B_2 \subseteq [N]$ and $a_0\geq 0$. Let $g$ be a positive integer with $g\leq \log_2 N$. If for each $b_1\in B_1$ and $b_2\in B_2$, there is a positive integer $x$ and a prime factor $p$ of $g$ such that $a_0+b_1+b_2=x^p$, then $$\min\{|B_1|,|B_2|\}\ll (\log N)^L,$$
where $L$ is an absolute constant, and the implied constant is absolute.
\end{prop}

A key ingredient is the following finite field model, which has appeared in \cite[Lemma 14.7]{DE15}.

\begin{lem}\label{lem:gpower}
Let $\eps>0$. Let $g$ be a positive integer and $p \equiv 1 \pmod g$ be a prime. If $A_p,B_p \subseteq \Z_p$ such that for each $a\in A_p$ and $b\in B_p$, there is $y\in \Z_p$ and a prime factor $q$ of $g$ such that $a+b=y^q$. Then $|A_p||B_p|\ll_{\eps} p^{1+\eps}$.
\end{lem}

Now we are ready to prove Proposition~\ref{prop:g}.
\begin{proof}[Proof of Proposition~\ref{prop:g}]
Let $A=a_0+B_1$ and let $B=B_2$. Then $A\subseteq a_0+[N]$ and $B \subseteq [N]$.

For each prime $p \equiv 1 \pmod g$, let $A_p$ be the image of $A$ modulo $p$ and view $A_p$ as a subset of $\Z_p$, and define $B_p$ similarly. Then by assumption, for each $a\in A_p$ and $b\in B_p$, there is $y\in \Z_p$ and a prime factor $q$ of $g$ such that $a+b=y^q$. Thus, Lemma~\ref{lem:gpower} implies that $|A_p||B_p|\ll_{\eps} p^{1+\eps}$. In particular, $\min \{|A_p|, |B_p|\}\ll p^{3/4}$.

By a quantitative version of Linnik's theorem (see, for example, \cite[Corollary 18.8]{IK04}), there exists an absolute constant $L$ such that if $Q\geq g^L$, then 
$$
\sum_{\substack{p \leq Q\\ p \equiv 1 \pmod g}} \log p\gg \frac{Q}{\phi(g)\sqrt{g}},
$$
where the implied constant is absolute and $\phi$ is Euler's totient function. Let $\mathcal{P}=\{p \equiv 1 \pmod g: p\leq Q\}$, where $Q=(100\log N)^{L+100}$. Since $g \leq \log_2 N$, we have $Q\geq g^L$. We partition the set $\mathcal{P}$ into two subsets:
$$
\mathcal{P}_A=\{p \in \mathcal{P}: |A_p|\leq |B_p|\}, \quad \mathcal{P}_B=\{p \in \mathcal{P}: |A_p|>|B_p|\}.
$$
Without loss of generality, we may assume that
$$
\sum_{p \in \mathcal{P_A}} \frac{\log p}{|A_{p}|} \geq \sum_{p \in \mathcal{P_B}} \frac{\log p}{|B_{p}|}.
$$
It follows from the quantitative Linnik's theorem that
\begin{align*}
\sum_{p \in \mathcal{P}} \frac{\log p}{|A_{p}|} \gg 
\sum_{p \in \mathcal{P_A}} \frac{\log p}{|A_{p}|}+\sum_{p \in \mathcal{P_B}} \frac{\log p}{|B_{p}|}
=\sum_{p \in \mathcal{P}} \frac{\log p}{\min \{|A_p|, |B_p|\}} 
\gg \sum_{p \in \mathcal{P}} \frac{\log p}{p^{3/4}}
\gg \frac{Q^{1/4}}{g^{3/2}}.
\end{align*}
Therefore, from the prime number theorem and Gallagher's larger sieve \cite{G71}, we have
$$
|A|\leq \frac{\sum_{p\in \mathcal{P}}\log p - \log N}{\sum_{p \in \mathcal{P}}\frac{\log p}{|A_p|}-\log N} \ll_{\eps} \frac{Q}{\frac{Q^{1/4}}{g^{3/2}}-\log N}\ll (\log N)^{L'},
$$
where $L'$ is an absolute constant.
\end{proof}

\subsubsection{Improved bounds under extra assumptions}\label{subsec:smallprimefactor}

In this section, we provide an improvement on Proposition~\ref{prop:a0+N} under some extra assumptions. We will see in Section~\ref{subsec:perfectpowerapply} that these assumptions are ultimately satisfied if $a_0$ has a ``small" prime factor. For a prime $p$ and a positive integer $n$, we use $v_p(n)$ to denote the largest integer $m$ with $p^m \mid n$ (equivalently, $p^m \mid \mid n$).

First, we consider the case where $a_0$ is even. 

\begin{prop}\label{prop:p=2}
There exists a positive constant $L$ with the following property: if $a_0\geq 0$ is even, and $B_1, B_2, \ldots, B_{13}\subseteq [N]$ such that $a_0+\sum_{i \in I} B_i \subset \PP$ for each $I \subseteq [13]$ with $|I| \in \{2,4\}$, then $\min_{i} |B_i|\ll (\log N)^L$, where the implied constant is absolute. 
\end{prop}
\begin{proof}
If $a_0=0$, this has been proved in \cite[Proposition 4.6]{DE15}. Next, assume that $a_0\geq 2$.    

Note that each element in $[N]$ can be uniquely written as $2^k(4\ell+\eps)$ where $k, \ell, \eps$ are integers with $0\leq k \leq \lfloor\log_2 N \rfloor$ and $\eps \in \{-1,1\}$.

We may assume that for each $i \in [13]$, there is $0\leq k_i\leq \log_2 N$ and $\eps_i \in \{-1,1\}$, such that all elements in $B_i$ are of the form $2^{k_i}(4\ell+\eps_i)$. Indeed, by the pigeonhole principle, we may pass each $B_i$ to a subset $B_i'$ with $|B_i'|\geq |B_i|/(4\log N)$ and the factor $\log N$ can be absorbed in our estimate. 

Applying pigeonhole again, by relabeling the indices, we may assume that there is $\eps_0 \in \{1,-1\}$, such that all elements in $B_i$ have the form $2^{k_i}(4\ell+\eps_0)$ for each $i\in [7]$. 

We claim that $k_i\geq 1$ for at least 4 different $i\in [7]$. Suppose otherwise, without loss of generality, we may assume $k_1=k_2=k_3=k_4=0$. If $4\mid a_0$, then for $b_1\in B_1$ and $b_2\in B_2$, we have $a_0+b_1+b_2 \equiv 2 \pmod 4$ and in particular $a_0+b_1+b_2$ is not a perfect power. Thus, we have $a_0\equiv 2 \pmod 4$. In that case, for $b_1\in B_1, b_2\in B_2, b_3\in B_3, b_4\in B_4$, we have $a_0+b_1+b_2+b_3+b_4 \equiv 2 \pmod 4$ and in particular $a_0+b_1+b_2+b_3+b_4$ is not a perfect power, a contradiction. This proves the claim. 

By the claim, without loss of generality, assume that $1\leq k_1\leq k_2\leq k_3\leq k_4$.  Let $b_1\in B_1, b_2\in B_2, b_3\in B_3, b_4\in B_4$.  Let $k=v_2(a_0)$. We consider the following cases:

\textbf{Case 1:} There exist $1<i<j$ such that $k,k_i,k_j$ are distinct. In this case, we have $v_2(a_0+b_i+b_j)=\min \{k,k_i,k_j\}\leq k_4$. 

Assume next \textbf{Case 1} does not hold. Then there are at most $2$ distinct elements in $\{k_1,k_2,k_3,k_4\}$. 

Then we have the following cases:

\textbf{Case 2a}: $k_1=k_2=k_3=k_4$. We have $v_2(b_1+b_2)=k_1+1$ and $v_2(b_1+b_2+b_3+b_4)\geq k_1+2$. If $k\neq k_1+1$, then we have $v_2(a_0+b_1+b_2)=\min \{k,k_1+1\}\leq k_1+1$; if $k=k_1+1$, then we have $v_2(a_0+b_1+b_2+b_3+b_4)=k_1+1$. 

\textbf{Case 2b}: there are precisely $2$ distinct elements in $\{k_1,k_2,k_3,k_4\}$ and $k \in \{k_1,k_2,k_3,k_4\}$. 

\begin{enumerate}
    \item If $k=k_i=k_j$ for some $1\leq i<j\leq 4$, then we have $v_2(b_i+b_j)=k_i+1$ and $v_2(a_0+b_i+b_j)=k_i$. 
    \item Otherwise, $k=k_i$ for some unique $i\in [4]$. Note that $k_1\neq k_4$ and $k \in \{k_1, k_4\}$. If $k=k_1$, then $k_1<k_2$ and thus $v_2(a_0+b_2+b_3)=k_1$; if $k=k_4$, then $k_3<k_4$ and thus $v_2(a_0+b_3+b_4)=k_3$. 
\end{enumerate}

In all of these cases, we can find a set $I \subseteq [7]$ with size 2 or $4$, such that there exists a constant $1\leq g\leq \log_2 N$, with the property that for each choice of $b_i\in B_i$, $v_2(a_0+\sum_{i\in I}b_i)=g$; since $a_0+\sum_{i\in I}b_i$ is a perfect power, this implies that $g\geq 2$ and there is a prime divisor $p$ of $g$ and a positive integer $x$ such that $a_0+\sum_{i\in I}b_i=x^p$. Thus, we can apply Proposition~\ref{prop:g} to get the required estimate.
\end{proof}

Next, we use a similar strategy to prove the following proposition.

\begin{prop}\label{prop:psmall}
There exists a positive constant $L$ with the following property: if
$p$ is an odd prime, $a_0$ is a positive integer with $p \mid a_0$, and $B_1, B_2, \ldots, B_{9}\subseteq [N]$ consisting of multiples of $p$, such that $a_0+\sum_{i \in I} B_i\subset \PP$ for each $I \subseteq [9]$ with $|I| \in \{2,3\}$, then $\min_{i} |B_i|\ll (\log N)^L$, where the implied constant is absolute. 
\end{prop}
\begin{proof}
Let $a_0=p^{u}(kp+\delta_0)$, where $u,k,\delta_0$ are integers with $1\leq \delta_0\leq p-1$. Similar to the proof of Proposition~\ref{prop:p=2}, we are only going to analyze $v_p(a_0+\sum_{i \in I} B_i)$, so without loss of generality, we may assume that $\delta_0=1$.

Note that each element $b \in [N]$ can be uniquely written as $p^u(\ell p+\delta)$ where $u, \ell, \delta$ are integers with $0\leq u \leq \lfloor\log_p N \rfloor$ and $1\leq \delta \leq p-1$. Consider the following three blocks: $C_1=\{1,2,\ldots, \frac{p-3}{2}\}$, $C_2=\{\frac{p-1}{2}\}$, and $C_3=\{\frac{p+1}{2}, \ldots, p-1\}$. We define $\psi(b)=j \in \{1,2,3\}$ if $\delta \in C_j$.  

We may assume that for each $i \in [9]$, there is $1\leq u_i\leq \log_2 N$ and $t_i \in \{1,2,3\}$, such that all elements in $B_i$ are of the form $p^{u_i}(kp+\delta)$ with $\psi(\delta)=t_i$. Indeed, by the pigeonhole principle, we may pass each $B_i$ to a subset $B_i'$ with $|B_i'|\geq |B_i|/(3\log_2 N)$ and the factor $\log N$ can be absorbed in our estimate. 

Next, we consider a few different cases.

\textbf{Case 1:} there exist at least 3 distinct numbers among $u_1,u_2,\ldots, u_{9}$. In this case, we can pick $1\leq i<j\leq 9$ such that $u,u_i,u_j$ are distinct. In this case, $v_p(a_0+b_i+b_j)=\min\{u,u_i,u_j\}\leq u_i$ for all $b_i\in B_i$ and $b_j\in B_j$.

\textbf{Case 2:} there are at most $2$ distinct numbers among $u_1,u_2,\ldots, u_{9}$. By the pigeonhole, we may assume that $u_1=u_2=\ldots=u_5$. 

\textbf{Case 2a:} $u \neq u_1$. By pigeonhole, we may assume that $\psi(B_1)=\psi(B_2)$. Let $b_1\in B_1$ and $b_2\in B_2$. Note that $v_p(b_1+b_2)=u_1$ and thus $v_p(a_0+b_1+b_2)=\min \{u,u_1\}\leq u_1$.

\textbf{Case 2b:} $u=u_1$.
\begin{enumerate}
    \item Suppose there are 3 $B_i$'s with $\psi(B_i)=2$ and $i\leq 5$, say $B_1, B_2, B_3$. Then we have $v_p(a_0+b_1+b_2+b_3)=u_1$ for all $b_1\in B_1$, $b_2\in B_2$, and $b_3\in B_3$. 
    \item Otherwise, by pigeonhole, we can find $1\leq i < j \leq 5$ with $\psi(B_i)=\psi(B_j) \in \{1,3\}$, say $B_1, B_2$. Then we have $v_p(a_0+b_1+b_2)=u_1$ for all $b_1\in B_1$ and $b_2\in B_2$.
\end{enumerate}
\medskip

In all of these cases, we can find a set $I \subseteq [7]$ with size 2 or $3$, such that there exists a constant $1\leq g\leq \log_2 N$, with the property that for each choice of $b_i\in B_i$, $v_p(a_0+\sum_{i\in I}b_i)=g$. Similar to the proof of Proposition~\ref{prop:p=2},  Proposition~\ref{prop:g} implies the required estimate.
\end{proof}

\subsection{Applications to Hilbert cubes in perfect powers}\label{subsec:perfectpowerapply}
We combine all the ingredients we have so far to prove Theorems~\ref{thm:localperfectpower} and~\ref{thm:ABCperfectpowerintro}.

\begin{proof}[Proof of Theorem~\ref{thm:localperfectpower}]
Consider the set $A=\{a_1,a_2,\ldots, a_d\}\subseteq [N]$. Since $a_0+\Sigma^*(A) \subset PP$ and all perfect powers are powerful, for each prime $p$, we can apply Proposition~\ref{prop:local} to those elements in $A$ which are not multiples of $p$.

Let $p$ be a prime factor of $a_0$ such that $p\leq (\log N)^{1-c}$. By Proposition~\ref{prop:local}(3), the number of elements in $A$ that are not multiples of $p$ is $\leq K'p\leq K'(\log N)^{1-c}$, where $K'$ is an absolute constant. Thus, we may assume that all the elements in $A$ are multiples of $p$, for otherwise we are done; indeed, we can remove all elements in $A$ that are not divisible by $p$ at a cost of at most $K'(\log N)^{1-c}$. The theorem follows from applying Theorem~\ref{main_strong} (c) (with Remark~\ref{rem:B_i} in mind) with $k=13$, $f(N)\ll(\log N)^{L}$ from Propositions~\ref{prop:p=2} and~\ref{prop:psmall}, and $g(N)\ll\log N$ from Lemma~\ref{lem:APperfectpower}(b). 
\end{proof}

\begin{proof}[Proof of Theorem~\ref{thm:ABCperfectpowerintro}]
(a) The proof is similar to that of Theorem~\ref{thm:localperfectpower}. 
We may assume that all elements $a_1,a_2,\ldots, a_d$ are multiples of $p$, for otherwise we are done by Proposition~\ref{prop:local}. The theorem follows from applying Theorem~\ref{main_strong}(a) (with Remark~\ref{rem:B_i} in mind) with $k=2$, $g(N)\ll \frac{\log \log N}{\log \log \log N}$  from Theorem~\ref{thm:APperfectpower}, and $f(N)\ll (\log N)^L$ from Propositions~\ref{prop:p=2} and~\ref{prop:psmall}.    

(b) By Theorem~\ref{main_strong}(a), we have $d\ll g(N)\log f(N)$ with $k=2$. The theorem follows from the following conditional bounds on $f(N)$ and $g(N)$. Under the ABC conjecture, Theorem~\ref{thm:APperfectpower}, and Proposition~\ref{prop:a0+N} imply that
$$ 
f(N)\leq \exp(O(\log \log N)^2), \quad g(N)\ll \frac{\log \log N}{\log \log \log N};
$$
under Conjecture~\ref{conj:LPS2}, Proposition~\ref{prop:a0+N} implies the stronger bound that $f(N)\ll (\log N)^L$.
\end{proof}

\section{Hilbert cubes in smooth numbers}\label{sec:smooth}
In this section, we prove the results mentioned in Section~\ref{sec:intro_smooth} about Hilbert cubes in $y(N)$-smooth numbers. Here we apply a mixture of our two frameworks. 

First, we use Theorems~\ref{thm:main4} and~\ref{thm:multi_incom} to deduce Corollaries~\ref{cor:smooth1} and~\ref{cor:smooth2}.
\begin{proof}[Proof of Corollary~\ref{cor:smooth1}]
Let $c>0$ be the constant from Theorem~\ref{thm:main4} with $\eps=1$. Assume $|A|\geq 10(y(N)/c)^{1/2}$, for otherwise we are already done. Let $B$ be the set of primes in $[y(N)+1, c|A|^2]$. Note that $c|A|^2 \geq 100y(N)$. Since $a_0+\Sigma^*(A)$ is contained in the set of $y(N)$-smooth integers, for any $b\in B,$ the members of $\Sigma^*(A)$ do not run over all residue classes modulo $b.$ It follows from Theorem~\ref{thm:main4} that
$$
y(N)^{1/2} \geq \log N \gg |A| \sum_{X< p\leq 100y(N)} \frac{\log p}{p},
$$
where $X$ is the smallest integer such that \[\sum_{y(N)< p\leq X}\log p \geq 2\log N.\] Obviously we must have $X\leq 10y(N).$ Hence
\[
\sum_{X< p\leq 100y(N)} \frac{\log p}{p} \geq \sum_{10y(N)< p\leq 100y(N)} \frac{\log p}{p} \gg 1.
\]

When $A$ is a multiset, let $B$ be the set of primes in $[y(N)+1,2y(N)]$. Since \[\sum_{p\in B}\log p\geq \frac{1}{10}y(N)>2\log N,\] it follows from Theorem~\ref{thm:multi_incom} that 
\[|A|\leq 2\max_{b\in B}b\leq 4y(N).\qedhere\]
\end{proof}

\begin{proof}[Proof of Corollary~\ref{cor:smooth2}]
We may assume that $|A|\geq \frac{\log N}{\log \log N}$, otherwise we are already done. Let $B$ be the set of primes in $[(\log N)^{2-\eps}, (\log N)^{2-\eps/2}] \subseteq [1, c|A|^2]$.  Since $a_0+\Sigma^*(A)$ is contained in the set of $y(N)$-smooth integers and $y(N)\leq (\log N)^{2-\eps}$, for any $b\in B,$ the members of $\Sigma^*(A)$ do not run over all residue classes modulo $b.$ It follows from Theorem~\ref{thm:main4} that
$$
\log N \gg |A| \sum_{X<p<(\log N)^{2-\eps/2}} \frac{\log p}{p},
$$
where $X$ is the smallest integer such that \[\sum_{(\log N)^{2-\eps}< p\leq X}\log p \geq 2\log N.\] Obviously we must have $X\leq 2(\log N)^{2-\eps}.$ Thus
\[
\sum_{X<p<(\log N)^{2-\eps/2}} \frac{\log p}{p} \geq \sum_{2(\log N)^{2-\eps}<p<(\log N)^{2-\eps/2}} \frac{\log p}{p} \gg \eps\log\log N.
\]

When $A$ is a multiset, let $B$ be the set of primes in $[y(N)+1,10y(N)]$. Since \[\sum_{p\in B}\log p > 2y(N)\geq 2\log N,\] it follows from Theorem~\ref{thm:multi_incom} that 
\[|A|\leq 2\max_{b\in B}b\leq 20y(N).\qedhere\]
\end{proof}

Finally, we present the proof of Theorem~\ref{thm:smooth1}. A main ingredient is the following result of Ruzsa~\cite{R92} about the greatest prime factor of a sumset. 
\begin{lem}[Ruzsa]\label{lem:R}
    Let $A,B$ be sets of integers, $|A|=|B|=k.$ Let $M=\max\{|a+b|:a\in A,b\in B\}.$ If $0\notin A+B,$ then
    \[P(A+B) \geq c\frac{k\log k}{\log M}\log\bigg(\frac{\log M}{\log k}\bigg)\]
    for some absolute constant $c>0.$
\end{lem}

\begin{proof}[Proof of Theorem~\ref{thm:smooth1}]
In view of Theorem~\ref{main_strong}(b), it suffices to show that $f(N)= (\log N)^{O(1)}$ and $g(N)\ll (\log N)^\alpha$ in the setting of Theorem~\ref{main_strong}. For any $B_1,B_2\subseteq[N],$ if $a_0+B_1+B_2$ is contained in $(\log N)^\alpha$-smooth integers, then from Lemma~\ref{lem:R} we get 
    \[(\log N)^\alpha>c\frac{\min(|B_1|,|B_2|)}{\log (a_0+2N)}.\]
    Since $a_0\leq N^K,$ we can take \[f(N)=\frac{2(2+K)}{c}(\log N)^{\alpha+1}.\] 

    Consider an arithmetic progression $Q=\{a+id:0\leq i\leq k-1\}\subseteq\N$ with $k\geq 3$. If $d=1,$ a classical result of Sylvester \cite{S1892} states that $P(Q)>k$ whenever $a\geq k+1$; if $d\geq 2,$ Shorey and Tijdeman~\cite{ST90} showed that $P(Q)>k$ with the only exception $(a,d,k)=(2,7,3).$ These results imply that the length of the longest arithmetic progression in $(\log N)^{\alpha}$-smooth numbers is bounded by $O((\log N)^{\alpha}).$ Thus we can take $g(N)\asymp (\log N)^\alpha.$
\end{proof}
%\begin{rem}
%    For the special case $\alpha=1,$ the value of $\delta$ we get is about $\frac{1}{3\cdot10^4}.$
%\end{rem}

We end the paper with some plausible thoughts on further improving Theorem~\ref{thm:smooth1} in the setting of sets. Whenever $y(N)\geq (\log N)^{\alpha}$ for some $\alpha>0,$ Lemma~\ref{lem:R} enables us to deduce an inverse theorem as a corollary of Theorem~\ref{thm:inverse}, and the bound on the rank gets better when $y(N)$ gets larger.
\begin{cor}\label{cor:smooth_inverse}
    Suppose $y(N) \geq (\log N)^\alpha$ for some $\alpha>0,$ $A\subseteq \N$ is a maximal multiset so that $a_0+\Sigma^*(A)\subseteq S\cap[N]$ for some nonnegative integer $a_0$, where $S$ is the set of $y(N)$-smooth numbers. Then, when $N$ is sufficiently large, there is a proper symmetric generalized arithmetic progression $Q$ of rank $r\leq 2(2\lceil1+\alpha^{-1}\rceil-1)$ such that $|Q\cap A| \ge |A|/2$ and $|Q| \ll |A|^{2\lceil1+\alpha^{-1}\rceil-r/2+0.2}.$  
\end{cor}
\begin{proof}
    In the above proof, we have shown that $f(N)\ll(\log N)^{\alpha+1}$ and $g(N)\asymp(\log N)^\alpha.$ It follows directly from Theorem~\ref{thm:inverse} with $k=2$ and $l=\lceil1+\alpha^{-1}\rceil$.
\end{proof}
In particular, if $y(N)\geq\log N,$ we get $r\leq 6$ when $A$ is a multiset and $r\leq 14$ when $A$ is a set; if $y(N)\geq N^c$ for some $c\in (0,1),$ we can essentially get $r\leq 2$ when $A$ is a multiset and $r\leq 6$ when $A$ is a set. 
We finish by noting that when $A$ is a set, a strong inverse theorem in the sense of Corollary~\ref{cor:smooth_inverse} with $r=1$ would give us a sharp bound on $|A|$. To see this, let $A\subseteq[N]$ be maximal so that $\Sigma^*(A)$ is contained in $\log N$-smooth numbers. Theorem~\ref{thm:smooth1} implies that $|A|\ll \log N$. By Lemma~\ref{lem:incomplete}, for each $p>\log N$, we have $|A_p|\ll \sqrt{p}$. Suppose there is a  symmetric arithmetic progression $Q$  such that $|Q\cap A|\ge |A|/2$ and $|Q|=O(|A|^{8-\frac{1}{2}})=O((\log N)^{15/2}).$ Then $Q=\{jd:-k\leq j\leq k\}$ for some $d\in\N$ with $k=O((\log N)^{15/2}).$ Now we can apply Gallagher's larger sieve \cite{G71} with $A\cap Q$ and $\mathcal{P}=\{\log N<p<100\log N:(p,d)=1\}$ to get
\[|A|\ll \frac{\underset{p\in \mathcal{P}}\sum\log p - \log k}{\underset{p \in \mathcal{P}}\sum\frac{\log p}{|A_p|}-\log k}\ll \frac{\underset{p\in \mathcal{P}}\sum\log p - \log k}{\underset{p \in \mathcal{P}}\sum\frac{\log p}{\sqrt{p}}-\log k}\ll (\log N)^{1/2}.\]
The key point here is that after taking out powers of $d$ from $\prod|a_i-a_j|$, we can use the fact that $Q$ is a short arithmetic progression and $d$ has few prime divisors of size $\gg \log N$.

\section*{Acknowledgments}
The third author thanks Ilya Shkredov for helpful discussions on $k$-th powers in arithmetic progressions, and Christian Elsholtz for sharing a copy of \cite{E03}. 

\bibliographystyle{abbrv}
\bibliography{main}

\end{document}